%% file: tamagawaNumbers.tex
\documentclass{amsart}
\usepackage[utf8]{inputenc}
\usepackage{amsmath}
\usepackage{amssymb}
\usepackage[pdfencoding=auto, psdextra]{hyperref}
\usepackage{amsthm}
\usepackage{dsfont}
\usepackage{tikz} 
\usepackage{adjustbox}
\usepackage{color}
\usepackage{mathrsfs}
\usepackage{xcolor}
\usepackage{url}
\usepackage{listings}
\usepackage{tikz-cd}
\usepackage{float}

\makeatletter
\newcommand*{\centerfloat}{%
  \parindent \z@
  \leftskip \z@ \@plus 1fil \@minus \textwidth
  \rightskip\leftskip
  \parfillskip \z@skip}
\makeatother

\lstdefinelanguage{Sage}[]{Python}
{morekeywords={False,sage,True},sensitive=true}
\definecolor{dblackcolor}{rgb}{0.0,0.0,0.0}
\definecolor{dbluecolor}{rgb}{0.01,0.02,0.7}
\definecolor{dgreencolor}{rgb}{0.2,0.4,0.0}
\definecolor{dgraycolor}{rgb}{0.30,0.3,0.30}

\DeclareFontFamily{U}{wncy}{}
\DeclareFontShape{U}{wncy}{m}{n}{<->wncyr10}{}
\DeclareSymbolFont{mcy}{U}{wncy}{m}{n}
\DeclareMathSymbol{\Sh}{\mathord}{mcy}{"58}
\makeatletter
\def\widebreve{\mathpalette\wide@breve}
\def\wide@breve#1#2{\sbox\z@{$#1#2$}%
     \mathop{\vbox{\m@th\ialign{##\crcr
\kern0.08em\brevefill#1{0.8\wd\z@}\crcr\noalign{\nointerlineskip}%
                    $\hss#1#2\hss$\crcr}}}\limits}
\def\brevefill#1#2{$\m@th\sbox\tw@{$#1($}%
  \hss\resizebox{#2}{\wd\tw@}{\rotatebox[origin=c]{90}{\upshape(}}\hss$}
\makeatletter

\newcommand{\res}{\mathbf{R}}
\newcommand{\gm}{\mathbb{G}_m}
\newcommand{\QQ}{\mathbb{Q}}
\newcommand{\ZZ}{\mathbb{Z}}
\newcommand{\T}{\mathbf{T}}
\newcommand{\Td}{\mathbf{T}_1}
\newcommand{\X}{\mathbf{X}^\star}
\newcommand{\La}{\Lambda}
\newcommand{\Ind}{\mathrm{Ind}}
\newcommand{\tr}{\mathrm{tr}}

\numberwithin{equation}{section}
\newtheorem{theorem}{Theorem}[section]
\newtheorem{lemma}[theorem]{Lemma}
\newtheorem{proposition}[theorem]{Proposition}
\newtheorem{corollary}[theorem]{Corollary}
\theoremstyle{remark}
\newtheorem{rem}[theorem]{Remark}
\newtheorem{example}[theorem]{Example}

\title[Explicit Tamagawa numbers of algebraic tori]{Explicit Tamagawa numbers for certain algebraic tori over number fields}
\author{Thomas R\"ud}
\begin{document}
\begin{abstract}
Given a number field extension $K/k$ with an intermediate field $K^+$ fixed by a central element of the corresponding Galois group of prime order $p$, we build an algebraic torus over $k$ whose rational points are elements of $K^\times$ sent to $k^\times$ via the norm map $N_{K/K^+}$. The goal is to compute the Tamagawa number  of that torus explicitly via Ono's formula that expresses it as a ratio of cohomological invariants. A fairly complete and detailed description of the cohomology of the character lattice of such a torus is given when $K/k$ is Galois. Partial results including the numerator are given when the extension is not Galois, or more generally when the torus is defined by an \'etale algebra. 

We also present tools developed in SAGE for this purpose, allowing us to build and compute the cohomology and explore the local-global principles for such an algebraic torus.

Particular attention is given to the case when $[K:K^+]=2$ and $K$ is a CM-field. This case corresponds to tori in $\mathrm{GSp}_{2n}$, and most examples will be in that setting. This is motivated by the application to abelian varieties over finite fields and the Hasse principle for bilinear forms.  
\end{abstract}
\maketitle

\section{Introduction}
\input{intro}

\textbf{Acknowledgement.}  I would like to thank Julia Gordon and Jeff Achter for suggesting this problem and constantly providing great interest and encouragement.

I am also immensely grateful to David Roe, who invited me to write the SAGE code used in this article during a coding sprint in August 2018 at the Institute for Mathematics and its Applications in Minneapolis. Since then, David has helped me format code for SAGE, and provided continuous support and advice.

Thank you to Wen-Wei Li with whom I co-authored the appendix of \cite{abvarcount} and briefly shared ideas. He notably introduced me to Kottwitz’s approach to this computation via the dual tori, and although it was not used in the paper, contributed in expanding my knowledge on the topic.

During the final days of proofreading this article, Chia-Fu Yu brought to my attention a project he led with similar goals based on \cite[Appendix A]{abvarcount} that overlap with this paper. His group computed Tamagawa numbers when $[K:K^+]=2$ (in the notations used above) and the field extention $K$ is cyclotomic, as well as a few other cases when $K$ is Galois and $\mathrm{Gal}(K/k) = \mathrm{Gal}(K/K^+)\times \mathrm{Gal}(K^+/k)$. Those results agree with the more general results of Proposition \ref{prop:sha:abelian} and Proposition \ref{prop:denom:iotanotin}. They also obtained the same bounds as Corollary \ref{cor:abeliannum} and examples of CM-fields with conclusions similar to Examples \ref{ex:d8} and \ref{ex:quat:computations}. Both contributions were done independently.

\section{Preliminaries}
\input{setup}\label{setup}
\subsection{Computation of the Tamagawa number for the auxilliary torus}\label{tder}
\input{tder}
\section{Tools for generic computations on algebraic tori in SAGE}
\input{sage}
\section{Transfer map as a counting function}
\subsection{Definition and properties}
\input{transferdef}
\subsection{Application to explicit computation of the first cohomology groups}
\input{explicit-h1}
\section{Computing the rest of the Cohomology groups}\label{subsec:rest:cohom}
\input{rest_cohom}
\section{Computation of the denominator}
\subsection{Generalities on the denominator}\label{denom:general}
\input{denom_generalities}
\subsection{Denominator in the abelian case}\label{denom:abel}
\input{abeliandenom}
\subsection{General case}\label{denom:fullgen}
\input{denom_rest}
\subsection{Results for 2-groups}\label{subsec:2groups}
\input{2groups}

\section{Extensions of the results}
\input{gen_case}
\newpage

\bibliographystyle{amsalpha}
\bibliography{biblio}
\end{document}

%% file: intro
The notion of a Tamagawa number as a geometric invariant of an algebraic group over a number field is now a fairly well understood topic. The Tamagawa number is defined as a volume of a certain fundamental domain with respect to a canonical measure. It is known that this volume is closely related to local-global principles and mass formulae. One of the early big contributions to this subject was a formula by Ono \cite{Ono} which computes the Tamagawa number $\tau(\T)$ of an algebraic torus. This formula was refined in \cite{voskresenski_book} into the formula $\displaystyle\tau(\mathbf{T}) = \frac{|\mathrm{Pic}(\mathbf{T})|}{|\Sh^1(\mathbf{T})|}$, where both invariants involved can be computed algebraically in the cohomology of the character lattice $\mathbf{X}^\star(\T)$ of $\T$. The formula can however be hard to evaluate in practical situations, for a general torus. The Tate-Shafarevich group, whose order is the the denominator of the formula, is famously hard to compute and depends heavily on the local structure of the splitting field of the torus. In this article, we evaluate the Tamagawa numbers for a particular class of algebraic tori that arise in situations of the following kind.

To a bilinear form over a number field $k$ we can associate an \emph{adjoint involution} on a $k$-algebra (see \cite{book_involution}). In the event that this $k$-algebra is a field $K$, the fixed points under the involution form a subfield $K^+$. One can look at the elements of $K^\times$ whose image under the norm map $N_{K/K^+}$ belongs to the base field $k$. Those elements in fact form the set of points of an algebraic torus and it turns out that  the Tate-Shafarevich group of this torus determines obstructions to the Hasse principle for the bilinear form we started with. Such a construction was made and explained in \cite{cortella} with $K$ being a CM-field and where an example of a torus with a nontrivial Tate-Shafarevich group was computed.

More generally, the motivation of this article was to give tools and results allowing one to compute the Tamagawa numbers of arbitrary maximal tori of the algebraic group $\mathrm{GSp}_{2n}$ arising as centralizers of regular semisimple elements.  Those tori are always split by an \'etale algebra with an involution fixing an index two subfield in each summand. While some of our results remain true in a more general setting, we focus mainly on the case where the involution is central in the absolute Galois group of $k$, in particular we give results for CM-\'etale algebras. The motivation for this comes from \cite{abvarcount}, which establishes  a mass formula for the size of the isogeny class of an ordinary principally polarized abelian variety over a finite field. This mass formula includes the Tamagawa number of the centralizer of the Frobenius element of the abelian variety, which is such a torus in $\mathrm{GSp}_{2n}$ splitting over a CM-\'etale algebra. The same torus and its Tamagawa number were also objects of interest in \cite{guo} dealing with class numbers of CM algebraic tori.

In order to conjecture some of the results presented in this article, it was very important to be able to compute the cohomology of those tori procedurally. Therefore, many tools in SAGE were implemented to define algebraic tori, including extensive methods to construct and study character lattices. The classes (in the sense of programming) of algebraic tori and $G$-lattices are to be part of a future release of SAGE. This makes defining character lattices via induction, morphisms, quotients, resolutions, sums, etc much easier and lets us compute their cohomology, and in many cases, the Tate-Shafarevich group. Most of the results conjectured using those tools have then led to proofs presented in this article, but some examples of Tate-Shafarevich groups are still only known by those computational methods. It is made clear in this article when an example or a result is only known via computation.

\subsection{Statement of the main results}


Throughout this article, $\mathbf{T}$ belongs to a class of algebraic tori including maximal tori of $\mathrm{GSp}_{2n}$. More precisely, $\T$ is obtained by the following construction:
\[\mathbf{T} = \mathrm{Ker}\left( \mathbb{G}_m\times_{\mathrm{Spec}(k)}\mathbf{R}_{K/k}(\mathbb{G}_m)\underset{(x, y)\mapsto x^{-1}N_{K/K^+}(y)}{\longrightarrow} \mathbf{R}_{K^+/k}(\mathbb{G}_m)\right),\]
where $K, K^+$ are as in \S \ref{setup}. Instead of focusing on a CM-field, we take any Galois field extension of the number field $k$ with intermediate field fixed by a central element of the corresponding Galois group of prime order.  In the later parts of the paper we allow $K$ to be non-Galois, or even more generally, an \'etale algebra which is a direct sum of such fields. We start by reformulating Ono's formula as $\tau(\T) = \frac{|\hat{H}^1(k, \mathbf{X}^\star(\T))|}{|\Sh^2(\mathbf{X}^\star(\T))|}$  and our results are based on a description of the structure of the character lattice for this class of tori.

It turns out that the transfer map (verlagerung) from the Galois group of the splitting field of the torus to the subgroup of elements fixing the intermediate extension $K^+$ is the key concept allowing us to relate the cohomology of $\mathbf{X}^\star(\T)$ to the one of an auxilliary torus, whose cohomology we compute in \S \ref{tder}. We therefore reintroduce this mapping in our precise setting with a slight twist.

We then apply the counting argument given by the transfer map to an explicit computation of the group $\hat{H}^1(k, \mathbf{X}^\star(\T))$, and prove the following in Theorem \ref{thm:h1nongal}:
\begin{theorem} Let $\T$ be a torus associated with an extension $K/k$ as above. Assuming that $K/k$ is Galois and $K^+$ is a subfield of $K$ fixed by a central subgroup of $\mathrm{Gal}(K/k)$ of prime order $p$, we have that $\hat{H}^1(k, \mathbf{X}^\star(\T))$ is trivial when the $p$-Sylow subgroups of $\mathrm{Gal}(K/k)$ are cyclic, and is isomorphic to $\ZZ/p\ZZ$ otherwise. 
\end{theorem}

Next we focus on the denominator of Ono's formula for this class of tori. This number depends on which subgroups of the Galois group arise as decomposition groups, but given a list of subgroups, we describe an algorithmic way to compute $\Sh^1(\T)$. In particular, if we only consider restriction maps to cyclic subgroups we obtain an invariant $\Sh^1_{\mathscr{C}}(\T)$ with $1\leq |\Sh^1(\T)|\leq |\Sh^1_{\mathscr{C}}(\T)|.$ This leads to the following results:

\begin{theorem}
Assuming that $K/k$ is Galois, let $G = \mathrm{Gal}(K/k)$, $N = \mathrm{Gal}(K^+/k)$, and let $G_p$ be a $p$-Sylow subgroup of $G$.  Then
\begin{itemize}
	\item If $\hat{H}^1(k, \mathbf{X}^\star(\T)) = 0$ (i.e. if $G_p$ is cyclic) then $\tau(\T) = \frac{1}{1}$. (Corollary \ref{cor:gpcyclic}+Theorem \ref{thm:h1h2}). 
	\item If $G$ is abelian then $\Sh^1_\mathscr{C}(\T) = \ZZ/p\ZZ$ if the $p$-Sylow subgroup of $G$ has a presentation $G_p = \ZZ/{p^{i_1}}\ZZ\times\cdots \ZZ/p^{i_n}\ZZ$ with $n>1$, $i_n>\max(i_1,\cdots i_{n-1})$, and $N$ is contained in the summand $\ZZ/p^{i_n}\ZZ$. Otherwise, $\Sh^1_\mathscr{C}(\T) = \Sh^1(\T) = 0$. (Proposition \ref{prop:sha:abelian}).
	\item If $G_p$ is not cyclic and $N$ is not contained in the commutator subgroup of $G$, then $\Sh^1_\mathscr{C}(\T) = \ZZ/p\ZZ$ if the $p$-Sylow subgroup of $G^{\mathrm{ab}}$ is cyclic or of the form described above. Otherwise $\Sh^1(\T) = \Sh^1_\mathscr{C}(\T) = 0$. (Proposition \ref{prop:denom:iotanotin}).
\end{itemize}
\end{theorem}
We also give a general description of $\Sh^1_\mathscr{C}(\T)$ and exhibit computations for Galois groups up to order 256. Then given elements $\alpha\in \Sh^1_\mathscr{C}(\T)\supset \Sh^1(\T)$, we describe which subgroups need to appear as decomposition groups need so that $\alpha\notin \Sh^1(\T)$, which in turn allows us to determine $\Sh^1(\T)$. In the last two cases described in the previous theorem, this description takes a simpler form, which we establish in Proposition \ref{prop:full:sha:abelian}.

In a final section we give an early approach to extending the results, in particular to non-Galois field extensions. First, we briefly comment on the case where $K^+$ is not necessarily fixed by a central element, but any normal cyclic subgroup of $\mathrm{Gal}(K/k)$. We also bring attention to the lack of known lower bound for the Tamagawa number in question, entertaining the possibility of this best lower bound being $0$. The major part of this section is dedicated to computing explicitly the numerator of the Tamagawa number for possibly non-Galois fields. For CM-fields, we give a complete description in terms of a condition that is easy to check computationally. Namely, in  \S \ref{sec:cm:nongal} we prove:
\begin{theorem}
Let $K/\QQ$ be a CM-field with Galois closure $K^\sharp$ and Galois group $G = \mathrm{Gal}(K^\sharp/\QQ)$. We have $\hat{H}^1(\QQ, \mathbf{X}^\star(\T))\subset \ZZ/2\ZZ$. Moreover, $\hat{H}^1(\QQ, \mathbf{X}^\star(\T)) = 0$ if and only if there is $g\in G$ such that $|\langle g\rangle \backslash G/N_2|$ is odd, where  $N_2 = \mathrm{Gal}(K^\sharp/K)$. 
\end{theorem}

In particular, we give the following data:
\begin{proposition}
Let $K/\QQ$ be a CM-field. Then 
\begin{itemize}
	\item If $[K:\QQ]=4$ then $\tau(\T)= 1$ unless $K/\QQ$ is Galois with Galois group $(\ZZ/2\ZZ)^2$, in which case $\tau(\T) = 2$. (Proposition \ref{prop:deg4})
	\item  If $[K:\QQ] = 6$ then $\tau(\T) = 1$. (Proposition \ref{prop:deg6})
	\item If $[K:\QQ] = 8$ then we list the values for $\hat{H}^1(\QQ, \mathbf{X}^\star(\T))$ and $\Sh_\mathscr{C}^1(\T)$ in \cite{comput_website}. 
\end{itemize}
\end{proposition}

Lastly, we consider the  case of CM-\'etale algebras and give some results, with intended applications to the formula in  \cite{abvarcount}.  We again give a somewhat elementary description of the numerator of Ono's formula, and show a few examples. Notably,  the Example \ref{ex:etale:power} shows that if  $K = K_1^{\oplus r}$ for some $r\geq 1$, $K_1$ a CM-field, and $K^+ = (K_1^+)^{\oplus r}$, if we define $\T^{K_1}$ to be the torus associated with $K_1$, then $\tau(\T) = 2^{r-1}\tau(\T^{K_1})$. This can lead to a contruction of tori with large Tamagawa numbers.

Throughout this paper, we used the LMFDB \cite{lmfdb} to create concrete examples of CM-fields and applied the results to compute the corresponding interesting Tamagawa numbers.

%% file: setup
\subsection{Definition of the tori}
Consider a finite Galois extension $K/k$ of number fields with intermediate field $K^+$ such that $[K:k] = pn$, and $[K:K^+]=p$ for some prime number $p$ and some integer $n$. Throughout this article we will let $G = \mathrm{Gal}(K/k)$, $N = \langle \iota\rangle = \mathrm{Gal}(K/K^+)$, and $H = \mathrm{Gal}(K^+/k)$. Furthermore, we assume $N$ to be central in $G$, which is automatically true when $p=2$. The notation $K^+$ comes from the main goal of this article, which is to give results related to the main theorem of \cite{abvarcount}. In that case, we have $k=\QQ$, and $K$ is a Galois CM-field with maximal totally real subfield $K^+$, in particular $p=2$. We have an exact sequence 
\begin{equation}\label{ses:group} 
1 \rightarrow N\rightarrow G\rightarrow H\rightarrow 1.
\end{equation}

Consider the algebraic torus of rank $(p-1)n+1$ defined by 
\[\mathbf{T} = \mathrm{Ker}\left( \mathbb{G}_m\times_{\mathrm{Spec}(k)}\mathbf{R}_{K/k}(\mathbb{G}_m)\underset{(x, y)\mapsto x^{-1}N_{K/K^+}(y)}{\longrightarrow} \mathbf{R}_{K^+/k}(\mathbb{G}_m)\right).\]

In particular, $\T(k) = \left\lbrace x\in K^\times : \prod_{\ell=0}^{p-1}\iota^\ell(x) \in k^\times\right\rbrace$.
\begin{example}\label{ex:CM-GSP}
Consider the case where $p=2$ and $K$ is a CM-field over $k=\QQ$, with maximal totally real subfield $K^+$.  We get $\T\subset \mathrm{GSp}_{2n}$. These tori arise in \cite{abvarcount} as centralizers of the Frobenius element corresponding to a principally polarized ordinary abelian variety. Also, in this case, $\iota$ is the Rosati involution. 
\end{example}

For any $G$-module $M$, we let $\hat{H}^i(G, M)=\hat{H}^i(K/k, M)$ denote its $i$-th Tate Cohomology group, and for all $i\in \mathbb{Z}$ we define \[\Sh^i(M) = \mathrm{Ker}\left(  \hat{H}^i(G, M)\rightarrow \prod_{\nu}\hat{H}^{i}(G_v, M)\right),\] where $\nu$ ranges over primes of $K$ and $G_v$ is the corresponding decomposition group (see \cite{platonov_rapinchuk}).

The goal of this article is to give a computation for the Tamagawa number $\tau(\T) = \tau_k(\T)$. We will not recall the definition of Tamagawa numbers, which can be found in the introduction of \cite{voskresenski}. To compute the latter we will focus on the formula given by the following theorem:

\begin{theorem}[\cite{Ono}, p.68]
\label{thm:ono_formula}
Let $\mathbf{T}$ be an algebraic torus over a number field $k$ and splitting over a Galois extension $L$. Then 
\[\tau(\mathbf{T}) =\frac{|\hat{H}^1(L/K,\mathbf{X}^\star(\mathbf{T})|}{|\Sh^1(\mathbf{T}(K))|}. \]
\end{theorem}

Using Tate-Nakayama duality (see \cite{platonov_rapinchuk}) one can rewrite the latter equality with $\Sh^1(\mathbf{T}(K)) = \Sh^2(\X(\T))$.

We now define an auxilliary torus of rank $(p-1)n$:
\[\Td = \mathrm{Ker}\left( \mathbf{R}_{K/\mathbb{Q}}(\mathbb{G}_m)\underset{N_{K/K^+}}{\longrightarrow} \mathbf{R}_{K^+/\mathbb{Q}}(\mathbb{G}_m)\right)= \res_{K^+/\QQ}\res_{K/K^+}^{(1)}(\mathbb{G}_m).\]
Here $\Td(k) = \left\lbrace x\in K^\times : \prod_{\ell=0}^{p-1}\iota^\ell(x) =1\right\rbrace$.

\begin{example} In the same setting as Example \ref{ex:CM-GSP} we have $\Td = \T\cap \mathrm{Sp}_{2n}$. Note that $\mathrm{Sp}_{2n}$ is the derived subgroup of $\mathrm{GSp}_{2n}$.
\end{example}

The two tori sit in the exact sequence 
\begin{equation}\label{ses:tori}1\rightarrow \mathbf{T}_1\rightarrow \T\rightarrow \gm\rightarrow 1.\end{equation}
\subsection{Character lattices}

Consider the group algebras $\ZZ[G]$ and $\ZZ[N]$ with respective augmentation ideals $J_G$ and $J_N$. By definition $N= \mathrm{Gal}(K/K^+)= \langle \iota\rangle$.

As $G$-modules, we have $\X(\Td) = \X(\res_{K^+/\QQ}\res_{K/K^+}^{(1)}(\mathbb{G}_m)) = \ZZ[G]/L_1$, where $L_1 = \{a\in \ZZ[G] : \iota a = a\} = \ZZ[G](1+\iota+\cdots+\iota^{p-1}) = \Ind_N^G(J_N)$.

The injection $\Td\subset \T$ yields a surjection $\X(\T)\rightarrow \X(\Td)$, and we get the description $\X(\T) = \ZZ/L$ where $L = L_1\cap J_G$.

We recover the corresponding exact sequences 
\begin{equation}\label{ses:characters} 0\rightarrow \ZZ\rightarrow \X(\T)\rightarrow \X(\Td)\rightarrow 0,
\end{equation}
and 
\begin{equation}\label{ses:characters_dual} 0\rightarrow L\rightarrow L_1\rightarrow \ZZ\rightarrow 0.
\end{equation}

For the sake of simplicity we will now write $\La$ and $\La_1$ to denote $\X(\T)$ and $\X(\Td)$ respectively.

%% file: tder
The cohomology of $\T_1=\res_{K^+/\QQ}\res_{K/K^+}^{(1)}(\mathbb{G}_m)$ and its character lattice $\La_1$ are very easy to compute, and will be useful for the rest of this article. 

\begin{proposition}\label{prop:tau_tder} We have $\hat{H}^i(G, \La_1) = 0$ if $i$ is even, and $\hat{H}^i(G, \La_1) = N$ if $i$ is odd.
\end{proposition}
\begin{proof} As a consequence of Shapiro's Lemma, we have \[\hat{H}^i(K/k, \T_1(K))=\hat{H}^i(K/K^+, \res_{K/K^+}^{(1)}(\mathbb{G}_m)(K)).\] Similarly, we get $\hat{H}^i(G, \La_1) = \hat{H}^i(N, \mathbf{X}^\star(\res_{K/K^+}^{(1)}(\mathbb{G}_m))$.   Note that we can write  $ \X(\res_{K/K^+}^{(1)}(\mathbb{G}_m)) = \ZZ[N]/\ZZ$ where we identify $\ZZ$ with its diagonal embedding in $\ZZ[N]$. Taking the cohomology of the short exact sequence 
\[0\rightarrow \ZZ\rightarrow \ZZ[N]\rightarrow \ZZ[N]/\ZZ \rightarrow 0,\]
the middle term being cohomologically trivial, we get  $\hat{H}^i(N, \ZZ[N]/\ZZ) =  \hat{H}^{i+1}(N, \ZZ)$. Since $N$ is cyclic, its cohomology is $2$-periodic, hence $\hat{H}^1(N, \ZZ[N]/\ZZ)  = \hat{H}^2(N, \ZZ) = \hat{H}^0(N, \ZZ)= N$ as desired.
\end{proof}

\begin{corollary} We have $\tau(\T_1) = p$.
\end{corollary}
\begin{proof}
We use Propositions \ref{prop:tau_tder} and \ref{thm:ono_formula}. The numerator is $|\hat{H}^1(G, \La_1)|=p$, and the denominator $\Sh^2(\La_1)$ is a subgroup of $\hat{H}^2(G,\La_1) = 0$, hence $\tau(\T_1)=\frac{p}{1}$.
\end{proof}

Note that since $N$ is cyclic, there will be primes in $K$ inert over $K^+$, and therefore, any torus over $K^+$ splitting over $K$ will have a trivial denominator in the formula of theorem \ref{thm:ono_formula}, since $N$ itself will appear as the decomposition group of an inert prime.

%% file: sage
Local-global principles, and more generally, cohomological invariants of algebraic tori are notoriously hard to compute directly outside of examples fitting in some nice short exact sequences such as norm-one tori with splitting field having particularly nice decomposition groups. In \cite{Ono}, the author proves that the Tamagawa number of $\res_{K/\QQ}^{(1)}(\mathbb{G}_m)$ where $K = \QQ(\sqrt{5}, \sqrt{29}, \sqrt{109}, \sqrt{281})$ is $\frac{1}{4}$. This specific extension is chosen because all its decomposition groups are cyclic, and it is abelian, which lets the author use Lyndon's formula (see \cite[p. 287]{lyndon1948}) to compute its cohomology groups.

More recently, in \cite{hoshi}, the authors have used GAP to compute cohomology of tori over local fields. In this paper, tori are studied through their character lattices with action of the Galois group of their minimal splitting field. The latter is seen as a finite subgroup of $\mathrm{GL}_n(\ZZ)$. This forces the user to input the action of the group as matrices, and also does not allow for someone to consider Galois group with possibly trivial action on the character lattice.

To ease the study of such objects, I implemented the classes of algebraic tori and $G$-lattices (to be seen as lattices of characters of tori) in SAGE. Those methods are to be added in a future release of SAGE and are presently available on my personal webpage with documentation. Here we include a brief description of some of the new SAGE methods with examples of their uses.

Examples such as Ono's can be recreated directly very easily. 

\adjustbox{scale = 0.7, center}{
\begin{lstlisting}
sage: L.<a, b, c, d> = NumberField([x^2-5, x^2-29, x^2-109, x^2-281])
sage: K = L.absolute_field('e')
sage: from sage.schemes.group_schemes.tori
....: import NormOneRestrictionOfScalars
sage: T = NormOneRestrictionOfScalars(K)
sage: T.Tamagawa_number()
1/4
\end{lstlisting}}

Moreover, the tools implemented for $G$-lattices provide many ways to create a lattice. We now show the construction of $\La$ and $\La_1$ for $G = Q_8$, and $N = Z(G)$.

\adjustbox{scale = 0.7, center}{
\begin{lstlisting}
sage: G = QuaternionGroup()
sage: N = G.center()
sage: Gm = GLattice(1); Gm
Ambient lattice of rank 1 with an action by a group of order 1
sage: IL = Gm.induced_lattice(G)
sage: Ld = IL.fixed_sublattice(N)
sage: L = Ld.zero_sum_sublattice()
sage: Lambda_d = IL.quotient_lattice(Ld); Lambda_d
Ambient lattice of rank 4 with a faithful action by a group of order 8
sage: Lambda = IL.quotient_lattice(L); Lambda
Ambient lattice of rank 5 with a faithful action by a group of order 8
\end{lstlisting}}

One can freely use direct sums, pullbacks, restrictions, duals, the norm map, and many more functions to create lattices. Then many methods have been implemented with cohomological uses, such as dimension-shifting, (co)flabby resolutions, restrictions. In particular, we have implemented a method to compute the Tate-Shafarevich groups $\Sh^i$ for $i=1,2$. For $i=1$ the program computes the restriction on $\hat{H}^1$ directly, whereas for $i=2$, since all cyclic subgroups of the Galois group appear as decomposition groups, we can use a flabby resolution of the lattice, to reduce the computation to the case $i=1$. More explicitly, given a group $G$ acting on a lattice $\Lambda$, we follow constructions made in \cite{hoshi} to compute a resolution 
\[0\rightarrow \Lambda \rightarrow P\rightarrow L\rightarrow 0, \]
where $P$ is \emph{permutation}, i.e. can be writte $P = \bigoplus_{i=1}^\ell \ZZ[G/H_i]$ for normal subgroups $H_1, \cdots, H_\ell$, and $L$ is \emph{flabby}, meaning $\hat{H}^{-1}(H, L) = 0$ for all subgroups $H\subset G$. Then by diagram chasing (see \cite[Lemma 2.9.1, Proposition 2.9.2]{multiplicative}), we get $\Sh^2(\Lambda) = \Sh^1(L)$. Moreover, if every decomposition group is cyclic, then $\Sh^2(\Lambda) = \hat{H}^1(G, L)$.

If the group associated to the lattice is the Galois group of a number field extension, then the algorithm will build every decomposition group, otherwise it will assume that every decomposition group is cyclic and the user can input the list of desired non-cyclic decomposition groups.

Continuing our example with $G = Q_8$.

\adjustbox{scale = 0.7, center}{
\begin{lstlisting}
sage: for i in range(-5, 6):
....:     print("H^"+str(i)+": Lambda:"
....: ,Lambda.Tate_Cohomology(i),", Lambda_d:",
....:  Lambda_d.Tate_Cohomology(i))
....:     
....:     
H^-5: Lambda: [] , Lambda_d: [2]
H^-4: Lambda: [4] , Lambda_d: []
H^-3: Lambda: [2] , Lambda_d: [2]
H^-2: Lambda: [2, 2] , Lambda_d: []
H^-1: Lambda: [2] , Lambda_d: [2]
H^0: Lambda: [4] , Lambda_d: []
H^1: Lambda: [2] , Lambda_d: [2]
H^2: Lambda: [2, 2] , Lambda_d: []
H^3: Lambda: [] , Lambda_d: [2]
H^4: Lambda: [4] , Lambda_d: []
H^5: Lambda: [2] , Lambda_d: [2]
\end{lstlisting}}

This confirms the computations of the cohomology of $\La_1$ done in the previous section, as well as computing the numerator of $\tau(\T)$, which is $|\hat{H}^1(G, \La)| = 2$. For the denominator, note that the subgroups of $Q_8$ are either cyclic or $Q_8$ itself. Therefore, if $Q_8$ appears as a decomposition group then $\Sh^1(\T) = 0$ and $\tau(\T) = 2$, otherwise  we have

\adjustbox{scale = 0.7, center}{
\begin{lstlisting}
sage: Sha = Lambda.Tate_Shafarevich_lattice(2); Sha
[2, 2]
\end{lstlisting}}
and so $\tau(\T)=\frac{1}{2}$.

Those methods have been used to compute the Tamagawa numbers of tori for every field extension up to degree $16$ and helped greatly with conjecturing the results proved in the rest of the article. We note that for some of these cases, no other method of finding the Tamagawa number is known.

%% file: transferdef
Let $\mathrm{tr} = \mathrm{tr}^G_N : G\rightarrow N$ denote the usual transfer map as defined in cite{rotman}. We will need some algebraic manipulations on the character lattice of $\T$ to build cocycles explicitely. We will use a counting functions directly related to $\tr$, therefore this section will be mostly proving results already known about $\tr$ in our setting, and no prior knowledge of transfer maps.

We start with a finite group $G$  and a central subgroup $N$ fitting in the short exact sequence 
\[1\rightarrow N\rightarrow G\rightarrow H = G/N\rightarrow 1.\]

We let $G_p, G^\mathrm{ab}$ denote a $p$-Sylow subgroup, and the abelianization of $G$ respectively. Write $|G_p| = p^r$ for some $r\in \mathbb{N}$.

For $g\in G$ we will write $\overline{g}=gN\in H$. For each coset $h\in H$, let $\hat{h}\in G$ be a representative.

Let $g\in G$ and $h\in H$. Since $h = \widehat{h}N$, we have $g\widehat{h}\in \widehat{\overline{g}h}N=\{\iota^i\widehat{\overline{g}h} : 1\leq i\leq p\}$.

Define the map $\psi : G\times H\rightarrow \mathbb{Z}/p\ZZ$ by  $g\widehat{h} = \iota^{\psi(g,h)}\widehat{\overline{g}h}$. We want to study the map $\varphi : G\rightarrow \ZZ/p\ZZ,\  g\mapsto \sum_{h\in H}\psi(g,h)$. It is clear from the definition of the transfer map that we have $\mathrm{tr} = g\mapsto \iota^{\varphi(g)}$, but we will show directly that it is well-defined and independent on the choice of representatives. We can immediately check from the definition that $\breve{\varphi}(1) = \hat{\varphi}(1) = 0$ and $\breve{\varphi}(\iota)=\hat{\varphi}(\iota) = |H|$, regardless of the choice of representatives.

\begin{lemma}\label{lem:indep:reps} $\varphi$ does not depend on the choice of representatives.
\end{lemma}
\begin{proof} Fix $h_0\in H$. Consider another choice of representative map ``1$\breve{}$'' such that $\breve{h_0} = \iota \hat{h_0}$ and $\breve{h} = \hat{h}$ for all $h\in H\backslash\{h_0\}$. The general case is obtained by repeating this process since $h = \hat{h}N$, every representative differ by a power of $\iota$. Write $\widehat{\psi}, \widebreve{\psi}, \hat{\varphi}, \breve{\varphi}$ the corresponding maps.  Let $g\in G\backslash\{1, \iota\}$. If $h\notin \{\overline{g}^{-1}h_0, h_0\}$ then $g\breve{h} = g\hat{h} = i^{\widehat{\psi}(g,h)}\widehat{\overline{g}h} = i^{\widehat{\psi(g,h)}}\widebreve{\overline{g}h}$, hence $\widehat{\psi}(g,h) = \widebreve{\psi}(g,h)$. Now observe that 
\[g\widebreve{h_0} = \iota g\widehat{h_0} = \iota \iota^{\hat{\psi}(g,h_0)} \widehat{\overline{g}h_0} =\iota^{\hat{\psi}(g,h_0)+1} \widebreve{\overline{g}h_0} , \]
and 
\[g\widebreve{\overline{g}^{-1}h_0} = g\widehat{\overline{g}^{-1}h_0} =  \iota^{\hat{\psi}(g,\overline{g}^{-1}h_0)} \widehat{h_0} =\iota^{\hat{\psi}(g,\overline{g}^{-1}h_0)} \iota^{-1}\widebreve{h_0}=\iota^{\hat{\psi}(g,\overline{g}^{-1}h_0)-1}\widebreve{h_0}.\]
Therefore, $\breve{\psi}(g, h_0) = \hat{\psi}(g, h_0)+1$ and  $\breve{\psi}(g, \overline{g}^{-1}h_0) = \hat{\psi}(g, \overline{g}^{-1}h_0)-1$, hence they compensate each other and $\hat{\varphi}(g) = \breve{\varphi}(g)$.
\end{proof}

\begin{rem}
For $p=2$, the function $\varphi$ counts the parity of the number of representatives of elements of $H$ which are not sent to another representative under the multiplication by $\iota$. 
\end{rem}

\begin{proposition} \label{prop:phi:homom} $\varphi$ is a group homomorphism.
\end{proposition}
\begin{proof}
 For the previous claim, notice that we have $\psi(g_1g_2, h) = \psi(g_1, \overline{g}_2h)+\psi(g_2,h)$ for all $g_1, g_2\in G$ and $h\in H$ . Summing this relation over $H$ yields the desired result.
\end{proof}

\begin{corollary}$\varphi$ factors through $(G^{\mathrm{ab}})_p$. If $p||H|$ (equivalently $r>1$) then  $N\subset \mathrm{Ker}(\varphi)$. 
\end{corollary}
\begin{proof} $\varphi$ is a group homomorphism and $\ZZ/p\ZZ$ is abelian, so the morphism factors through $G^\mathrm{ab}$. Also, by virtue of $\varphi$ being a homomorphism, if $g$ has order coprime to $p$ then its image is $0$. 

For all $h\in H$ we have $\psi(\iota, H) = 1$ by definition. So $\varphi(\iota) = |H|$ hence $\varphi(N)= 0$ when $p||H|$.
\end{proof}

\begin{lemma}\label{lem:cyclicgroup:surj} If $G = G_p$ is cyclic,  then $\varphi$ is surjective.
\end{lemma}
\begin{proof} Write $G = \langle \varepsilon\rangle$, with $\varepsilon^{p^{r-1}} = \iota$. We make the choice $\widehat{\varepsilon^iN}= \varepsilon^j$ where $0\leq j\leq p^{r-1}-1$. Clearly $\varepsilon \widehat{\varepsilon^iN}= \widehat{\varepsilon^{i+1}N}$ if $0\leq i\leq p^{r-1}-1$, and $\varepsilon  \widehat{\varepsilon^{ p^{r-1}N}} = \iota $, so $\varphi(\varepsilon)=1$.
\end{proof}

\begin{lemma}\label{lem:iotaincyclic} Let $g\in G$. If  $\iota\in \langle g\rangle$ then $\varphi(g) = |\langle g\rangle\backslash G|$.
\end{lemma}
\begin{proof}
Note that $\iota\in \langle g\rangle$ therefore $\langle g\rangle\backslash G\cong \langle\overline{g}\rangle\backslash H = \bigsqcup_{i} \langle\overline{g}\rangle h_i $. Consequently $\langle g\rangle\backslash G = \bigsqcup_{i}\langle{g}\rangle \hat{h_i}$. Left multiplication by $g$ preserves each right coset, so by the same computation as in Lemma \ref{lem:cyclicgroup:surj} we have $\sum_{h\in \langle\overline{g}\rangle h_i}\psi(g, h) = 1$, hence $\varphi(g) = |\langle\overline{g} \rangle\backslash H| = |\langle g \rangle\backslash G|$.
\end{proof}
\begin{corollary} If $G_p$ is cyclic, then $\varphi$ is surjective. 
\end{corollary}
\begin{proof}
Let $g$ be a generator of $G_p$. By Lemma \ref{lem:iotaincyclic} $\varphi(g) = |\overline{G_p}\backslash H| = |G_p\backslash G|$, which is coprime to $p$, hence $\varphi$ is surjective.
\end{proof}

\begin{proposition} \label{prop:trans_cyclicness} $\varphi$ is the zero map if and only if $G_p$ is noncyclic.
\end{proposition}
\begin{proof} We have already shown that $\varphi$ is surjective if $G_p$ is cyclic. Now assume that $G_p$ is not cyclic. Let $g\in G$. We want to show $\varphi(g) = 0$.

\begin{itemize}
	\item \emph{Case 1.} Assume $\iota\in \langle g\rangle$. By Lemma \ref{lem:iotaincyclic} we have $\varphi(g) = |\langle g\rangle\backslash G|$, which is still divisible by $p$ since $\langle g \rangle$ is cyclic and therefore cannot be a $p$-Sylow subgroup. So $\varphi(g) = 0$.
	\item \emph{Case 2.} Assume $\iota \notin \langle g \rangle$. This means that the sets $g^iN$ are all distinct sets for $0\leq i \leq \ell$ where $\ell$ is the order of $g$. Decompose $\langle \overline{g}\rangle H = \bigsqcup_{i}\langle \overline{g}\rangle h_i$.  By Lemma \ref{lem:indep:reps} we are free to pick representatives, so we pick $\widehat{\overline{g^i}h_j} = g^i\hat{h_j}$ for some representative $\hat{h_j}$ of $h_j$. Multiplication by $g$ preserves all the cosets, and \[g \widehat{\overline{g^i}h_j}= g g^i\hat{h_j} =g^{i+1}\hat{h_j} = \widehat{\overline{g^{i_1}}h_j}.\]
	Therefore $\psi(g, \langle \overline{g}\rangle h_i) = 0$ for all $h_i$ hence $\varphi(g) = 0$.
\end{itemize}
\end{proof}

\begin{corollary}\label{cor:cyclic:complement} Let $G$ be a finite group with central subgroup $N$ of order $p$. If $G_p$ is cyclic, then $G_p$ has a unique subgroup of order $|G_p \backslash G|$. In particular this subgroup is normal and $G_p$ is its complement. 
\end{corollary}
\begin{proof} The idea is to take the kernel of $\varphi$, and repeat the process to the kernel, and so on, until we get to a group whose $p$-Sylow is just $N$, and take its complement. 

We will proceed by induction on $r$. Recall that $|G_p| = p^r$. If $r=1$ then $G_p= N \cong G/\mathrm{Ker}(\varphi)$, and $\mathrm{Ker}(\varphi)$ has therefore a complement by Schur-Zassenhaus Theorem, so we can write $G = \mathrm{Ker}(\varphi)\rtimes N$.
 Assume now that  $G_p$ is cyclic and $r>1$. Let $m = |G_p \backslash G|$. We have a surjective group homomorphism $\varphi: G\rightarrow \mathbb{Z}/p\mathbb{Z}$, let $M$ denote its kernel. We know that $|M| = |G|/p = p^{r-1}m$, so we can use our induction hypothesis to conclude that $M$ contains a unique normal subgroup of size $m$, call it $C$. By uniqueness, $C$ is a characteristic subgroup of $K$ (stable under any automorphism), and $M$ is normal in $G$, therefore $C$ is normal in $G$. 
\end{proof}

\begin{rem}
The assumption that $N$ is central is necessary. For example, the dihedral group $D_{18}$ has a normal subgroup of order $3$, but no normal group of order $2$.
\end{rem}

%% file: explicit-h1
In this section we apply the results of the previous section to compute $\hat{H}^1$ directly. In \S \ref{subsec:rest:cohom} we give a shorter, albeit more abstract proof. 

For $\lambda\in \ZZ[G]$, let $[\lambda]$ denote its image in the quotient $\La = \ZZ/L$, and $\{\lambda\}$ denote the $1$-cocycle defined by $\{\lambda\}_g = g[\lambda] - [\lambda]$. 
\begin{lemma}\label{lem:cobounds} For all $g\in G$ we have  $\sum_{\ell =0}^{p-1}\{\iota^\ell g\} =0$.
\end{lemma}
\begin{proof} Such a coboundary is immediately values in $L$ by definition of $L$ as sublattice of $\ZZ[G]^{N}$ of zero-sum vectors.
\end{proof}

For each $h\in H$ fix a choice of preimage $\hat{h}\in G$, and for the sake of convenience, we choose $\hat{1}=1$. 

Note that the cohomology of the exact sequence (\ref{ses:characters}) gives 
\[ 0=\hat{H}^1(G, \ZZ)\rightarrow \hat{H}^1(G, \La)\rightarrow \hat{H}^1(G, \La_1)=N.\]
The equality on the right was proved in Proposition \ref{prop:tau_tder}. Therefore, $\hat{H}^1(G, \La)$ embeds as a subgroup of $N\cong \ZZ/p\ZZ$, hence one only needs to find one nontrivial $1$-cocycle to determine $\hat{H}^1(G, \La) = N$.
\begin{theorem}\label{thm:numerator}
 Define the coboundary $b = \{\sum_{i=0}^{p-1}i\sum_{h\in H} \iota^i\hat{h}\}$. If $G_p$ is not cyclic then $b = p a$ where $a$ is a nontrivial $1$-cocycle, in particular $\hat{H}^1(G,\Lambda) = N$.
\end{theorem}
\begin{proof}
 Recall that for all $g\in G$ we have $g\widehat{h} = \iota^{\psi(g,h)}\widehat{\overline{g}h}$. For $g\in G$ and  $0\leq k\leq p-1$ we define $I_{g}^k = \{h\in H : \psi(g, h) = k\}$. It is clear that $\bigsqcup_{0\leq k\leq p-1}I_g^k = H$. In particular, we have that $p$ divides $\sum_{k=0}^{p-1}k |I_g^k| = \sum_{h\in H}\psi(g,h)$ for all $g\in G$. Also note that in $\Lambda$, for all $g\in G$ we have $[g\sum_{i=0}^{p-1}{\iota^i}] = [ \sum_{i=0}^{p-1}{\iota^i}]$.

We only want to know the result modulo $p$, so the following computation will be done in $\Lambda/p\Lambda$.
\begin{align*}
b_g &= g\left(\sum_{i=0}^{p-1}i\sum_{h\in H} [\iota^i\hat{h}]\right)- \sum_{i=0}^{p-1}i\sum_{h\in H} [\iota^i\hat{h}]\\
& =\sum_{i=0}^{p-1}i\sum_{k=0}^{p-1}\sum_{h\in I_g^k} [\iota^{i+k}\widehat{gh}]-\sum_{i=0}^{p-1}i\sum_{k=0}^{p-1}\sum_{h\in I_g^k} [\iota^{i}\widehat{gh}]\\
& =\sum_{k=0}^{p-1}\sum_{h\in I_g^k} \sum_{i=0}^{p-1}i[\iota^{i+k}\widehat{gh}]-\sum_{k=0}^{p-1}\sum_{h\in I_g^k}\sum_{i=0}^{p-1}i [\iota^{i}\widehat{gh}]\\
& =\sum_{k=0}^{p-1}\sum_{h\in I_g^k} \left(\sum_{i=0}^{p-1}i[\iota^{i+k}\widehat{gh}]- \sum_{i=0}^{p-1}i[\iota^{i}\widehat{gh}]\right).\\
\end{align*}
We are working modulo $p$ and $\iota$ has order $p$ so we have that $\sum_{i=0}^{p-1}i[\iota^{i+k}\widehat{gh}] = \sum_{i=0}^{p-1}(i-k)[\iota^{i-k}\widehat{gh}]$. Therefore,
\begin{align*}
b_g & =\sum_{k=0}^{p-1}\sum_{h\in I_g^k} \left(\sum_{i=0}^{p-1}(i-k)[\iota^{i}\widehat{gh}]- \sum_{i=0}^{p-1}i[\iota^{i}\widehat{gh}]\right)\\
& =\sum_{k=0}^{p-1}\sum_{h\in I_g^k} \sum_{i=0}^{p-1}-k[\iota^{i}\widehat{gh}]=\sum_{k=0}^{p-1}-k\sum_{h\in I_g^k} \sum_{i=0}^{p-1}[\iota^{i}\widehat{gh}]\\
&= \sum_{k=0}^{p-1}-k|I_g^k| \left(\sum_{i=0}^{p-1}[\iota^{i}]\right)=-\left(\sum_{i=0}^{p-1}[\iota^{i}]\right)\underset{=\varphi(g) = 0}{\underbrace{\sum_{k=0}^{p-1}k|I_g^k|}}=0.
\end{align*} 
We have proved that each coordinate of $b_g$ is a multiple of $p$ for all $g\in G$, hence $a = \frac{b}{p}$ is a $1$-cocycle of $G$ with coefficients in $\Lambda$. We are left to show that $a$ is not a coboundary. Since $\sum_{\ell=0}^{p-1}\{\iota^\ell\}=0$ by Lemma \ref{lem:cobounds}, we can generate all coboundaries with $\{\{\iota^\ell \hat{h}\} : \ell\in \ZZ/p\ZZ\ \ell\neq 1\ h\in H\}$. 

We now mimic the computation above to compute $b_{\iota^\ell}$.

\begin{align*}
b_{\iota^\ell} &= \sum_{i=0}^{p-1} \sum_{h\in H} i[\iota^{i+\ell}\hat{h}]-\sum_{i=0}^{p-1} \sum_{h\in H} i[\iota^{i}\hat{h}]=  \sum_{h\in H} \left(\sum_{i=0}^{p-1}i[\iota^{i+\ell}\hat{h}]-\sum_{i=0}^{p-1} i[\iota^{i}\hat{h}]\right)\\
 &=  \sum_{h\in H} \left(\sum_{i=\ell}^{p+\ell-1}(i-\ell)[\iota^{i}\hat{h}]-\sum_{i=0}^{p-1} i[\iota^{i}\hat{h}]\right)\\
 &=  \sum_{h\in H} \left(-\sum_{i=0}^{\ell-1}i[\iota^i\hat{h}] -\sum_{i=\ell}^{p-1}\ell[\iota^{i}\hat{h}]+\sum_{i=0}^{\ell-1} (i+p-\ell)[\iota^{i}\hat{h}]\right)\\
  &=-\sum_{h\in H} \left(\sum_{i=\ell}^{p-1}\ell[\iota^{i}\hat{h}]+\sum_{i=0}^{\ell-1} (\ell-p)[\iota^{i}\hat{h}]\right)= \sum_{h\in H} \left(\sum_{i=0}^{\ell-1} p[\iota^{i}\hat{h}]\right) - \ell|H|\sum_{i=0}^{p-1}[\iota^i] .\\
\end{align*}

Therefore, $a_{\iota^\ell} = \sum_{h\in H} \left(\sum_{i=0}^{\ell-1} [\iota^{i}\hat{h}]\right) - \ell\frac{|H|}{p}\sum_{i=0}^{p-1}[\iota^i]$. In particular, using that for all $h\in H$ we have $[\iota \hat{h}] = [\iota]+\sum_{i\neq 1}([\iota^i]-[\iota^i h]) = \sum_{i=0}^{p-1}[\iota^i] - \sum_{i\neq 1}[\iota^i\hat{h}]$. 
\[a_{\iota} = \sum_{h\in H} [\iota\hat{h}] - \ell\frac{|H|}{p}\sum_{i=0}^{p-1}[\iota^i]= \sum_{h\in H}\sum_{i\neq 0} [\iota^i\hat{h}] - |H|(\frac{\ell}{p}-1)\sum_{i=0}^{p-1}[\iota^i].\]
Fix $\hat{h}\neq 1$. If $2\leq \ell\leq p-1$ then $\{\iota^\ell \hat{h}\}_{\iota}= [\iota^{\ell+1}\hat{h}]-[\iota^\ell\hat{h}]$. Now for $\ell = 0$ we have \[\{\hat{h}\}_{\iota} = [\iota \hat{h}] - [\hat{h}] = \sum_{i=0}^{p-1}[\iota^i] - 2[\hat{h}] - \sum_{i = 2}^{p-1}[\iota^i \hat{h}].\]	
The $[\iota^\ell\hat{h}]$-coefficient of $a_\iota$ for $\ell\neq 1$ is always $1$. The goal is to now use that fact for each $h\in H$. Because of this, if $a$ is a sum of coboundaries, it must contains summands spanned by $\{\{\iota^\ell\hat{h}\} : \ell\neq 1\}$. However, each of those coboundaries evaluated at $\iota$ are of the form $\{\iota^{\ell+1}\hat{h} - \iota^\ell \hat{h}\}$, hence summing them will simplify by a telescopic argument, and there is only one possibility to get the desired described coefficients for $a_\iota$.

Assuming $a$ is a coboundary, then assume $k\{\hat{h}\}$ is a summand for some $k\in \mathbb{N}$. On order to have a coefficient $1$ at $[\hat{h}]$ the only possibility is to add $(2k+1)\{\iota^{p-1}\hat{h}\}$, but the evaluation of the latter at $\iota$ has $[\iota^{p-1}\hat{h}]$-coefficient $-(2k+1)$. Since that coefficient must also be $1$, and the only other generating coboundary having nonzero $[\iota^{p-1}\hat{h}]$-coordinate is $\{\iota^{p-2}\hat{h}\}$, it means that  $3k+2$ $\{\iota^{p-2}\hat{h}\}$ must also be a summand. We can repeat the process for each power of $\iota$, and we determine that for each $h$  the cocycle $a$ must have a summand of the form $\sum_{i=0}^{p-1}((i+1)k+i)\{\iota^{p-i}\hat{h}\}$. Using Lemma \ref{lem:cobounds}, we know that  $\sum_{i=0}^{p-1}\{\iota^i \hat{h}\} = 0$, hence this summand can be written as
\begin{align*}
\sum_{i=0}^{p-1}((i+1)k+i)\{\iota^{p-i}\hat{h}\} &=\sum_{i=0}^{p-1}((i+1)k+i)\{\iota^{p-i}\hat{h}\} -\sum_{i=0}^{p-1}\{\iota^i \hat{h}\} \\
&= \sum_{i=0}^{p-1}(ik+i)\{\iota^{p-i}\hat{h}\} = (k+1)\sum_{i=0}^{p-1}i\{\iota^{i}\hat{h}\} 
\end{align*}
However in the basis we picked for $\Lambda$, looking modulo $p$, we have that 
\begin{align*}
\sum_{i=0}^{p-1}i \{\iota^i \hat{h}\}_{\iota} &= \sum_{i=0}^{p-1} i [\iota^{p-i+1}\hat{h}] -\sum_{i=0}^{p-1} i [\iota^{p-i}\hat{h}] \text{ (mod $p$)}\\
&\equiv \sum_{i=0}^{p-1} (i+1) [\iota^{p-i}\hat{h}] -\sum_{i=0}^{p-1} i [\iota^{p-i}\hat{h}]\equiv \sum_{i=0}^{p-1} [\iota^{p-i}\hat{h}] \equiv \sum_{i=0}^{p-1}[\iota^i]\text{ (mod $p$)}\\
\end{align*}
Therefore, in our basis for $\Lambda$, modulo $p$, our coboundary $(k+1)\sum_{i=0}^{p-1}i\{\iota^{i}\hat{h}\}$ does not have any $[\iota^\ell \hat{h}]$-coordinate, for $\ell \neq 1$, hence $a$ cannot be a sum of those generating coboundaries, which yields $H^1(G,\Lambda)\neq 0$.
\end{proof}

\begin{lemma}\label{lemma:g:cyclic} If $G$ is cyclic, then for all $i$ we have $\hat{H}^i(G,\La_1) = \left\lbrace\begin{array}{l} 0 \text{ if }i\text{ is odd}\\ G/N \text{ if }i\text{ is even}\end{array}  \right.$.
\end{lemma}
\begin{proof} Since $G$ is cyclic, its cohomology is $2$-periodic hence we only need to solve it for $i=0,1$.
Using the definition $\Lambda = \ZZ[G]/L$ we have $\hat{H}^i(G, \Lambda) = \hat{H}^{i+1}(G, L)$. The action of $G$ on $L$ factors through $G/N$, and $L$ is isomorphic to the augmentation ideal of $\ZZ[G/N]$, so we have the short exact sequence 
\[0\rightarrow L \rightarrow \mathrm{Ind}_N^G(\mathbb{Z})=\ZZ[G/N]\rightarrow \ZZ \rightarrow 0.\]
Since it is an augmentation ideal, $L$ has no $G$-fixed point, so $\hat{H}^0(G, L)=0$. Taking the cohomology of the sequence above, using Shapiro Lemma, and Hilbert 90, we get 
\[ 0\rightarrow \underset{= N}{\underbrace{\hat{H}^0(N, \mathbb{Z})}}\rightarrow \underset{=G}{\underbrace{\hat{H}^0(G, \mathbb{Z})}}\rightarrow \hat{H}^1(G, L)\rightarrow \underset{=0}{\underbrace{\hat{H}^1(N, \mathbb{Z})}}.\] 
We used the cyclicity of $G$ to write $\hat{H}^0(G, \mathbb{Z}) = \ZZ/|G|\ZZ = G$. This finishes the proof.
\end{proof}
\begin{proposition}\label{prop:cyclic:h1}
If $G_p$ is cyclic, then $\hat{H}^1(G, \La) = 0$.
\end{proposition}
\begin{proof}
Assume $G_p$ is cyclic. Using Corollary \ref{cor:cyclic:complement}, we can write $G = M\rtimes G_p$ for some subgroup $M$ of order coprime to $p$. The corresponding inflation-restriction exact sequence is:
\[0\rightarrow H^1(G_p, \Lambda^{M}) \rightarrow H^1(G,\Lambda) \rightarrow H^1(M, \Lambda)^{G_p} \rightarrow H^2(G_p,\Lambda^{M}) \rightarrow H^2(G,\Lambda).\]  
Again, $H^1(G, \Lambda)\subset N$ so in particular, it is $p$-torsion, but $H^1(M, \Lambda)$ is $|M|$-torsion, hence the map $H^1(G, \Lambda)\rightarrow H^1(M, \Lambda)^{G_p}$ is trivial. Therefore, $H^1(G_p, \Lambda^{M})= H^1(G,\Lambda)$.

Look at the short exact sequence of $M$-modules
\[0\rightarrow L \rightarrow \ZZ[G] \rightarrow \Lambda\rightarrow 0.\]
Taking the group cohomology, we get 
\[0\rightarrow L^M \rightarrow \ZZ[G/M]= \ZZ[G_p] \rightarrow \Lambda^M\rightarrow H^1(M, L)\rightarrow 0,\]
which gives us the short exact sequence 
\begin{equation}0\rightarrow \ZZ[G_p]/L\cap \ZZ[G_p] \rightarrow \Lambda^M\rightarrow H^1(M, L)\rightarrow 0.\label{eq:kfixed}\end{equation}
Importantly, since $M$ is normal in $G$, every term in this sequence has a $G/M = G_p$-module structure. Notably, fixed elements of $H^i(M, L)$ by $G_p$ corresponds to the image of the restriction map $H^i(G, L)\rightarrow H^i(M, L).$

Now, let $\Lambda_p = \ZZ[G_p]/L\cap \ZZ[G_p]$. This is the same construction as for one $\Lambda$, but replacing $G$ by $G_p$: in particular we can use \ref{lemma:g:cyclic} to compute its cohomology. Taking the Tate-cohomology of (\ref{eq:kfixed}), we get 
\[\hat{H}^{i-1}(G_p, H^1(M, L))\rightarrow \hat{H}^i(G_p, \Lambda_p)\rightarrow \hat{H}^i(G_p, \Lambda^M)\rightarrow \hat{H}^{i}(G_p, H^1(M, L)).\]
We have that $H^1(M, L)$ is $|M|$-torsion, and $|M|$ is coprime to $p$, hence for all $i$ we have $\hat{H}^{i}(G_p, H^1(M, L))=0$. Consequently, $\hat{H}^i(G_p, \Lambda_p)=\hat{H}^i(G_p, \Lambda^M)$, and using Lemma \ref{lemma:g:cyclic} we obtain $\hat{H}^i(G_p, \Lambda^M) = 0$, as desired.
\end{proof}

\begin{rem} Note that in the proof above we do not necessarily have $\Lambda^M = \ZZ[G]^M / L^M$. Indeed, the group $H^1(M, L)$ might not be trivial. For example, when $G = \ZZ/24\ZZ$, we have $M = \ZZ/3\ZZ$ and $H^1(M, L) = M$. 
\end{rem}

\begin{rem}
The proof of Proposition \ref{prop:cyclic:h1} shows more generally that whenever $G$ is $p$-nilpotent, then $\hat{H}^1(G, \Lambda) = \hat{H}^1(G_p, \Lambda_p)$ (with the notations of the proof).
\end{rem}

%% file: rest_cohom
We now give a more abstract application of the transfer map to compute all cohomology groups. The following commutative diagram is just a reformulation of the short exact sequences (\ref{ses:characters}) and (\ref{ses:characters_dual}).

\begin{figure}[H]
\centering
\begin{tikzcd}
& & 0\arrow[d]& 0\arrow[d]& & \\
& \displaystyle 0 \arrow[r] &\displaystyle  L \arrow[r]\arrow[d]& \displaystyle L_1\arrow[r]\arrow[d]& \displaystyle \ZZ\arrow[r]&  \displaystyle 0\\
&  \displaystyle 0 \arrow[r] &\displaystyle  \ZZ[G]\arrow[r]\arrow[d]& \displaystyle \ZZ[G]\arrow[d]\arrow[r]& \displaystyle 0& &\\
\displaystyle 0\arrow[r] &  \displaystyle \ZZ \arrow[r] &\displaystyle  \Lambda \arrow[d]\arrow[r]& \displaystyle \La_1\arrow[r]\arrow[d]& \displaystyle 0& &\\
& & 0& 0& & \\
\end{tikzcd}
\end{figure}
This diagram commutes and has both exact rows and columns. Note that an element of $\ZZ$ in the top right can be lifted to an element of $L_1$, which imbeds into $\ZZ[G]$. It can be sent to the left copy of $\ZZ[G]$ and into the quotient $\Lambda$, this element is however trivial in the quotient $\La_1$ so by exactness of the sequence below, it belongs to the image of the injection $\ZZ\rightarrow \Lambda$. This gives us a map $\ZZ\rightarrow \ZZ$ which is just the identity. 


Let $S$ be a subgroup of $G$. We want to compute $\hat{H}^i(S, \Lambda)$. Since $\ZZ[G]$ is cohomologically trivial, as direct sum of induced $S$-modules, we get the following exact sequence on cohomology.

\adjustbox{scale = 0.7, center}{
\begin{tikzcd}[column sep=small]
\displaystyle\hat{H}^{i-1}(S, \La_1)\ar[equal]{d}\arrow[r]&\displaystyle\hat{H}^i(S, \ZZ)\ar[equal]{d}\arrow[r]&\displaystyle\hat{H}^i(S, \Lambda)\ar[equal]{d}\arrow[r]&\displaystyle\hat{H}^i(S, \La_1)\ar[equal]{d}\arrow[r]&\displaystyle\hat{H}^{i+1}(S, \ZZ)\ar[equal]{d}\arrow[r]&\displaystyle\hat{H}^{i+1}(S, \Lambda)\ar[equal]{d}\arrow[r]&\displaystyle\hat{H}^{i+1}(S, \La_1)\ar[equal]{d}\\
\displaystyle\hat{H}^{i}(S, L_1)\arrow[r]&\displaystyle\hat{H}^i(S, \ZZ)\arrow[r]&\displaystyle\hat{H}^{i+1}(S, L)\arrow[r]&\displaystyle\hat{H}^{i+1}(S, L_1)\arrow[r]&\displaystyle\hat{H}^{i+1}(S, \ZZ)\arrow[r]&\displaystyle\hat{H}^{i+2}(S, L)\arrow[r]&\displaystyle\hat{H}^{i+2}(S, L_1)\\
\end{tikzcd}}

Note that we will not prove the commutativity of this diagram because it is not needed. Each vertical arrow is an isomorphism, so we will compute each element in the bottow row and this will give us the cohomology groups on the top row.

\begin{lemma}\label{lem:ld:mod}
For any subgroup $S\leq G$ we have that $L_1\cong \mathrm{Ind}_{N}^{NS}(\ZZ)^{[G:NS]}$ as an $S$-module.
\end{lemma}
\begin{proof}
Let $S' = NS$ and decompose $G$ into the right cosets $G = \sqcup_{i}S'g_i$ and so $\ZZ[G]$ is the $\ZZ$-span of $\ZZ[S'g_i]$, each of them being an $S'$-module. Each summand is isomorphic to $\ZZ[S']$ , hence 
\[L_1 \cong \{ a\in \ZZ[S'] : \iota a=a\}^{[G:S']}\cong \ZZ[S'/N]^{[G:S']}\cong \mathrm{Ind}_{N}^{S'}(\mathbb{Z})^{[G:S']}.\] 

\end{proof}
\begin{corollary}\label{cor:cohom:ld} Let $S\leq G$ be a subgroup. For all $i\in \ZZ$ we have 
\[\hat{H}^{2i}(S, L_1) = \left\lbrace\begin{array}{l}N^{[G:S]}\text{ if }\iota\in S\\ 0\text{ else} \end{array}\right. ,\text{ and}\  \hat{H}^{2i+1}(S, L_1) = 0 .\]
\end{corollary}
\begin{proof}
If $\iota\notin S$ then $L_1 \cong \mathrm{Ind}_{N}^{SN}(\ZZ)^{[G:SN]}=\mathrm{Ind}_{1}^{S}(\ZZ)^{\frac{[G:S]}{p}}$, hence is induced. If $\iota\in S$ then $\hat{H}^i(S, L_1) = \hat{H}^i(N, \ZZ)^{[G:S]}$. Since $N$ is cyclic, we only need to compute $\hat{H}^0(N, \ZZ)= N$ and $\hat{H}^1(N, \ZZ) = 0$.
\end{proof}

Plug this in the second row of the previous diagram. If we pick $i$ odd then we have
\begin{equation}\label{eq:cohom:L}0 \rightarrow \hat{H}^i(S, \ZZ)\rightarrow \hat{H}^{i+1}(S, L)\rightarrow \hat{H}^{i+1}(S, L_1)\overset{\varphi_i}{\rightarrow} \hat{H}^{i+1}(S, \ZZ)\rightarrow \hat{H}^{i+2}(S, L)\rightarrow 0.\end{equation}
\begin{corollary}\label{cor:iotanotins} Let $S\leq G$ be a subgroup such that $\iota\notin S$. Then for all $i\in \ZZ$ we have $\hat{H}^{i+1}(S, L)=\hat{H}^{i}(S, \ZZ)$. Consequently $\hat{H}^i(S, \Lambda) = \hat{H}^i(S, \ZZ)$.
\end{corollary}
\begin{proof} This is a direct consequence of plugging in the results of Corollary \ref{cor:cohom:ld} in (\ref{eq:cohom:L})
\end{proof}
\begin{lemma}\label{lem:transfer:generic} If the transfer map $\mathrm{tr}^G_N : G\rightarrow N$ is trivial, then $\hat{H}^1(G, \Lambda) = N$ and $\hat{H}^2(G, \Lambda) = G^{\mathrm{ab}}$, otherwise $\hat{H}^1(G, \Lambda) = 0$ and $\hat{H}^2(G, \Lambda) = G^{\mathrm{ab}}/N$.  This result is still true if $N$ is not assumed to be a central subgroup of $G$.
\end{lemma}
\begin{proof} The sequence (\ref{eq:cohom:L}) with $S=G$ and $i=1$ is 
\[0  \rightarrow \hat{H}^{2}(G, L)\rightarrow N\overset{\varphi_1}{\rightarrow} G^{\mathrm{ab}}\rightarrow \hat{H}^{3}(G, L)\rightarrow 0.\]
We used Hilbert $90$ to write $\hat{H}^1(G, \ZZ) = 0$. Using the short exact sequence 
\[0\rightarrow \ZZ\rightarrow \QQ\rightarrow \QQ/\ZZ\rightarrow,0\]
with trivial action of $G$ (and $N$), we note that the middle term is cohomologically trivial (it is uniquely divisible), hence we can write
\[\hat{H}^2(G, L_1) =\hat{H}^2(G, \ZZ[G/N])  = \hat{H}^2(N, \ZZ)=\hat{H}^1(N, \QQ/\ZZ) =\mathrm{Hom}(N, \QQ/\ZZ)= N.\]
Likewise $\hat{H}^2(G, \ZZ) = \hat{H}^1(G, \QQ/\ZZ) = \mathrm{Hom}(G, \QQ/\ZZ) = G^\mathrm{ab}$. Classically (see \cite[Chapter III, section 9]{brown}), the corresponding map $\hat{H}^2(N, \ZZ)\rightarrow \hat{H}^2(G, \ZZ)$ is induced by the transfer map.
Therefore, the map $\varphi_1$ is really a map on the dual groups $\mathrm{Hom}(N, \QQ/\ZZ)\rightarrow \mathrm{Hom}(G, \QQ/\ZZ)$ defined by $\alpha \mapsto \alpha\circ\mathrm{tr}^G_N$. We have $\hat{H}^1(G, \Lambda)= \hat{H}^2(G, L) = \mathrm{Ker}(\varphi_1)$. The order of $N$ being prime, we have $\mathrm{tr}^G_N$ is either trivial or surjective. In the former case $\hat{H}^1(G, \Lambda) = N$, else $\hat{H}^1(G, \Lambda) = 0$. Plugging this back into the equation above, we get the corresponding result for $\hat{H}^2(G, \Lambda)= \hat{H}^3(G, L)$. 

The assumption of $N$ being central in $G$ has only been used to compute the triviality of the transfer map in terms of the structure of a $p$-Sylow subgroup of $G$. This section has only used the fact that $N$ is normal so far. Thus, we do not need the assumption that $N$ is central here.
\end{proof}
\begin{lemma}\label{lem:transfer:equal}
Let $S\leq G$ with $\iota\in S$. If the transfer map $\mathrm{tr}^S_N : S\rightarrow N$ is trivial, then $\hat{H}^i(S, \Lambda) = \hat{H}^i(S, \ZZ)$ for $i$ even and $\hat{H}^i(S, \Lambda)/\hat{H}^i(S,\ZZ) = N^{[G:S]}$ for $i$ odd.
\end{lemma}
\begin{proof}
The map $\varphi_i$ of equation (\ref{eq:cohom:L}) is the natural map from $\hat{H}^{i}(S, \mathrm{Ind}_N^S(\ZZ))^{[G:S]}= \hat{H}^i(N, \ZZ)$ to $\hat{H}^i(S, \ZZ)$. This is exactly the corestriction map on cohomology, induced by the transfer map. We refer again to the section on the transfer map in \cite[Chapter III, section 9]{brown}. Therefore, if the transfer map is trivial, $\varphi_i$ is trivial, hence we get the two exact sequences: 
\[0 \rightarrow \hat{H}^i(S, \ZZ)\rightarrow \hat{H}^{i+1}(S, L)\rightarrow \hat{H}^{i+1}(S, L_1)\rightarrow 0,\]
and
\[0 \rightarrow \hat{H}^{i+1}(S, \ZZ)\rightarrow \hat{H}^{i+2}(S, L)\rightarrow 0,\]
which yields the desired result by replacing $\hat{H}^i(S, L)$ by $\hat{H}^{i-1}(S, \Lambda)$.
\end{proof}

With all this preparation, we can now give results for the cohomology groups involved in the computation of the Tamagawa number. 

\begin{theorem}\label{thm:h1h2} Let $S\leq G$. 
\begin{itemize}
\item If $\iota\notin S$ then $\hat{H}^1(S, \Lambda) = 0$ and $\hat{H}^2(S, \Lambda) = S^{\mathrm{ab}}$. 
\item If $\iota\in S$ and $S_p$ is cyclic, then $\hat{H}^{1}(S, \Lambda) = N^{[G:S]-1}$ and $\hat{H}^2(S, \Lambda) = S^{\mathrm{ab}}/N$.
\item If $\iota\in S$ and $S_p$ is not cyclic, then $\hat{H}^{1}(S, \Lambda) = N^{[G:S]}$ and $\hat{H}^2(S, \Lambda) = S^{\mathrm{ab}}$.
\end{itemize}
\end{theorem}
\begin{proof} By Hilbert 90, we have $\hat{H}^1(S, \ZZ) = 0$. Taking the cohomology of the exact sequence
\[0\rightarrow \ZZ\rightarrow \QQ\rightarrow \QQ/\ZZ\rightarrow 0,\]
since the middle term is uniquely divisible, we have $\hat{H}^2(S, \ZZ) = \hat{H}^1(S, \QQ/\ZZ) = \mathrm{Hom}(S, \QQ/\ZZ) = S^{\mathrm{ab}}$.

The exact sequence (\ref{eq:cohom:L}) with $i=1$ becomes 
\begin{equation}\label{eq:h1}
0\rightarrow \hat{H}^2(S, L)\overset{\xi}{\rightarrow} \hat{H}^2(S, L_1)\overset{\varphi_1}{\rightarrow} S^{\mathrm{ab}}\rightarrow \hat{H}^3(S, L)\rightarrow 0.
\end{equation} 

By Corollary \ref{cor:iotanotins} and Lemma \ref{lem:transfer:equal} we get the first and last cases. 

If $S$ is cyclic, then $\hat{H}^2(S, L) = \hat{H}^0(S, L) = \{\{a_i\}\in N^{[G:S]} : \prod a_i = 1\}= N^{[G:S]-1}$. Therefore, the cokernel of $\xi$ of (\ref{eq:h1}) is isomorphic to $N$. Since we have 
\[0 \rightarrow \mathrm{Coker}(\xi)\rightarrow S \rightarrow\hat{H}^3(S, L) \rightarrow 0,\]
as desired.

Now assume $S_p$ is cyclic. We can follow the proof of Proposition \ref{prop:cyclic:h1} verbatim. Let $K_S=K\cap S$ be the complement of $S_p$ in $S$.  We get $\hat{H}^1(S, \Lambda) = \hat{H}^1(S_p, \Lambda^{K_S})$ and look at the cohomology of 
\[0\rightarrow L\rightarrow \ZZ[G]\rightarrow \Lambda\rightarrow 0.\]
This time $\ZZ[G]^{K_S} = \ZZ[G_p]^{[K:K_S]}$. Still following the proof of $\ref{prop:cyclic:h1}$,  the same cohomological sequence yields $\hat{H}^1(S_p,\Lambda^{K_S})\cong \hat{H}^1(S_p,\ZZ[G_p]^{[K:K_S]}/L^{K_S}) = \hat{H}^2(S_p,L^{K_S}) = \hat{H}^0(S_p,L^{K_S})$. The latter corresponds to sum zero elements of $\hat{H}^0(S_p, L_1^{K_S}) = N^{[G_p:S_p][K:K_S]} =N^{[G:S]}$. Therefore, $\hat{H}^1(S, \Lambda) = N^{[G:S]-1}$, and we can then follow the same argument as the cyclic case above.
\end{proof}

%% file: denom_generalities
The goal of this section is to give computations of the denominator of the Tamagawa number as stated in Theorem \ref{thm:ono_formula}. 

\begin{lemma}\label{lem:h1torsion}
The group $\hat{H}^1(K/k, \mathbf{T}(K))$ is $p$-torsion, and therefore so is $\Sh^1(\mathbf{T}(K))$.
\end{lemma}
\begin{proof}
We look at the cohomology of the short exact sequence (\ref{ses:tori}) and get 
\[\hat{H}^1(K/k, \T_1(K))\rightarrow \hat{H}^1(K/k, \mathbf{T}(K))\rightarrow \hat{H}^1(K/k, \mathbb{G}_m(K)) = 0.\]
Therefore, we have a surjection $\hat{H}^1(K/k, \T_1(K))\rightarrow \hat{H}^1(K/k, \mathbf{T}(K))$, so it suffices to prove that $\hat{H}^1(K/k, \T_1(K))$ is $p$-torsion.

Now $\hat{H}^1(K/k, \T_1(K)) = \hat{H}^1(K/k, \res_{K^+/\QQ}\res_{K/K^+}^{(1)}(\mathbb{G}_m)(K))$ which in turns equals $\hat{H}^1(K/K^+, \res_{K/K^+}^{(1)}(\mathbb{G}_m)(K))$ by Shapiro's Lemma.

Now looking at the cohomology of the short exact sequence 
\[1\rightarrow \res_{K/K^+}^{(1)}(\mathbb{G}_m)\rightarrow \res_{K/K^+}(\mathbb{G}_m)\rightarrow \mathbb{G}_m\rightarrow 1, \]
we get $\hat{H}^1(K/K^+, \res_{K/K^+}^{(1)}(\mathbb{G}_m)(K)) = \hat{H}^0(K/K^+, \mathbb{G}_m(K)) = {(K^+)}^\times/N_{K/K^+}(K^\times),$ which is clearly $p$-torsion, since $N_{K/K^+}(K^\times)$ contains $N_{K/K^+}((K^+)^\star) = (({K^+})^\times)^p$.
\end{proof}

By Tate-Nakayama duality (see \cite[p.307]{platonov_rapinchuk}), we know that $\Sh^1(\mathbf{T}(K)) = \Sh^2(\mathbf{X}^\star(\mathbf{T}))$.

\begin{proposition}\label{prop:bounds} We have 
\[\frac{p}{|G^\mathrm{ab}[p]|}\leq\tau(\mathbf{T})\leq p.\]
\end{proposition}
\begin{proof} We know that the numerator of Ono's formula is $|\hat{H}^1(G, \Lambda)|\leq p$, and the denominator is a subgroup of $\hat{H}^2(G, \Lambda) = G^\mathrm{ab}$, and is $p$-torsion, so it is contained in $G^\mathrm{ab}[p]$. 
\end{proof}

\begin{proposition}\label{prop:cyclic:denom} If $G_p$ is cyclic, then $\Sh^1(\mathbf{T}(K)) = 0$.
\end{proposition}
\begin{proof}
By Chebotarev density Theorem, every cyclic subgroup of $G$ appears as the decomposition group at unramified places. In particular, so does $G_p$. Therefore, we know that $\Sh^2(\Lambda)\subset \mathrm{Ker}(\hat{H}^2(G, \Lambda)\rightarrow \hat{H}^2(G_p, \Lambda))$. By restriction-corestriction, we know that the composite map 
\[\hat{H}^2(G, \Lambda)\rightarrow \hat{H}^2(G_p, \Lambda)\rightarrow \hat{H}^2(G, \Lambda)\]
is multiplication by $[G:G_p]$. Since the image of $\Sh^2(\Lambda)$ through the restriction map is trivial, we get that $\Sh^2(\Lambda)$ is annihilated by $[G:G_p]$, which is coprime to $p$. Therefore, $\Sh^2(\Lambda)$ is both $p$-torsion and $[G:G_p]$-torsion, and is hence trivial.
\end{proof}

\begin{corollary}\label{cor:gpcyclic} If $G_p$ is cyclic, then $\tau(\mathbf{T}) = \frac{1}{1}=1$.
\end{corollary}
\begin{proof}
Proposition \ref{prop:cyclic:h1} tells us that the numerator is $1$ whereas Proposition \ref{prop:cyclic:denom} gives us the triviality of the denominator. 
\end{proof}

If $G_p$ is not cyclic, the answer depends on ramification groups and therefore depends on the particular field $K$. We will give an explanation of the computation in general, with some bounds and examples. 

Let $\mathscr{C}$ denote the set of cyclic subgroups of $G$, and let $\mathscr{D}$ be the set of subgroups of $G$ not in $\mathscr{C}$ arising as decomposition groups of $K/k$. In particular, subgroups in $\mathscr{D}$ can only occur at some ramified primes.  If $\mathscr{S}$ is a collection of subgroups of $G$ then we let 
\[\Sh^i_{\mathscr{S}}(\Lambda) = \mathrm{Ker}\left(\hat{H}^i(G, \Lambda)\rightarrow \prod_{S\in \mathscr{S}}\hat{H}^i(S, \Lambda)\right).\]

\begin{lemma}
If $G\in \mathscr{S}$ or $G_p\in \mathscr{S}$, then $\Sh_\mathscr{S}^2(\Lambda) = 0$.I
\end{lemma}
\begin{proof}
The first case is immediate. For the second case, consider the restriction-corestriction maps 
\[\hat{H}^1(G, \T) \rightarrow \hat{H}^1(G_p, \T)\rightarrow \hat{H}^1(G, \T),\]
the composition of the two maps is multiplication by the index of $G_p$, which is coprime to $p$. Since $H^1(G, \T)$ is $p$-torsion, this is an isomorphism, hence the restriction map $\hat{H}^1(G, \T)\rightarrow \hat{H}^1(G_p, \T)$ is an injection.
\end{proof}

Also, we clearly have 
\[\Sh^2(\Lambda) = \Sh^2_{\mathscr{C}\cup \mathscr{D}}(\Lambda) =\Sh^2_{\mathscr{C}}(\Lambda) \cap \Sh^2_{\mathscr{D}}(\Lambda)\subset \Sh^2_{\mathscr{C}}(\Lambda).   \]

We will be now focusing on computing $\Sh^2_{\mathscr{C}}(\Lambda)$. By Theorem \ref{thm:h1h2}, assuming $G_p$ is not cyclic, we can rewrite this as 
\[\Sh^2_\mathscr{C}(\Lambda) = \mathrm{Ker}\left( G^{\mathrm{ab}} \rightarrow \prod_{\underset{\iota\notin \langle\alpha\rangle}{\alpha\in G}}\langle \alpha \rangle \times  \prod_{\underset{\iota\in \langle\alpha\rangle}{\alpha\in G}}\langle \alpha \rangle/N   \right).\] 

We know that $\Sh^2_\mathscr{C}(\mathbb{Z}) = \Sh^1(\mathbb{G}_m) = 0$, in particular, $G^\mathrm{ab}\rightarrow \prod_{\alpha\in G}\langle \alpha\rangle$ is an injection. 

For $\alpha \in G$, the morphism of groups $G^\mathrm{ab}\rightarrow \langle \alpha\rangle$ is not canonical, it is a map on the Pontryagin's duals of the groups, it is the map $\mathrm{Hom}(G, \QQ/\ZZ)\rightarrow \mathrm{Hom}(\langle \alpha\rangle, \QQ/\ZZ)$ given by restriction.  By virtue of $\Sh^2(\Lambda)$ being $p$-torsion we can restrict the computation to elements $\alpha$ of $G^\mathrm{ab}[p]$.

We denote the isomorphism $G^\mathrm{ab} \rightarrow \mathrm{Hom}(G, \QQ/\ZZ)$ by $g\mapsto t_{g}$. This morphism depends on the choice of presentation of $G^{\mathrm{ab}}$. We can rewrite $\mathrm{Ker}(G^\mathrm{ab} \rightarrow \langle \alpha\rangle) = \{g\in G^{\mathrm{ab}} : t_{g}(\alpha) = 0\}$.

\begin{example}\label{ex:quat:computations}
If $p=2$ and $G$ is the quaternion group  $Q_8$, then all proper subgroups of $G$ are cyclic, and all nontrivial subgroups contain $\iota$. Write $G = \langle \alpha, \beta\rangle$ with $\alpha^2 = \beta^2  = \iota$, $\beta\alpha\beta^{-1} = \alpha^{-1}$. Then $\mathrm{G}^{\mathrm{ab}} = (\ZZ/2\ZZ)^2.$ Every element of $G^{\mathrm{ab}}$ is $2$-torsion, and hence is sent to $2$-torsion elements of $\langle \alpha\rangle \times \langle\beta\rangle$, which is the subgroup $N\times N$. However, $\Sh^2_\mathscr{C}(\Lambda) = \mathrm{Ker}\left(G^{\mathrm{ab}}\rightarrow \langle \alpha\rangle/N\times \langle \beta\rangle/N  \right)$. By our previous point, every element is mapped into $N\times N$, and hence belongs to this kernel. Consequently $\Sh^2(\Lambda)$ is trivial if and only if $G$ appears as decomposition group, else $\Sh^2(\Lambda) = (\ZZ/2\ZZ)^2$.

Let $K$ be given by the polynomial $x^{8} + 68 x^{6} + 986 x^{4} + 4624 x^{2} + 4624$. Using the LMFDB \cite{lmfdb}, we know that this is a CM field with Galois group $Q_8$,  discriminant $2^{22}\cdot 17^{6}$, and both decomposition groups at ramified primes are cyclic, isomorphic to $\ZZ/4\ZZ$. Therefore, in this case, $\tau(\mathbf{T}) = \frac{2}{4} = \frac{1}{2}$. This is the smallest example of non-integral Tamagawa number for the tori we are interested in. If $|G|\leq 8$ and $G\neq Q_8$, then $\tau(\mathbf{T})\in \{1, 2\}$.
\end{example}

%% file: abeliandenom
Here $G = G^{\mathrm{ab}}$ and $G_p$ is not cyclic. We can assume $G = G_p$. Indeed, $G$ will decompose in a direct sum of its Sylow subgroups, and cyclic subgroups contained in $q$-Sylow subgroups never contain $N$ when $p\neq q$. Therefore, the map $G_p\rightarrow \prod_{g\in G_\ell}\langle g\rangle$ is always trivial and does not contribute to $\Sh^2_\mathscr{C}(\Lambda)\subset G[p]\subset G_p$.

Write down $G = \bigoplus_{i=1}^r \left(\ZZ/p^i\ZZ\right)^{n_i}$ with $n_r\neq 0$ and let $\mathscr{C}=\mathscr{C}_0\sqcup \mathscr{C}_1$ be the collection of cyclic subgroups of $G$, where $\mathscr{C}_0$ is the collection of cyclic groups containing $N$. We can write $g\in G$ as $g = \sum_{i=1}^r \vec{\mathbf{m}}_i$ where $\vec{\mathbf{m}}_i\in \left(\ZZ/p^i\ZZ\right)^{n_i}$.	Let us denote the isomorphism $G\rightarrow \mathrm{Hom}(G,\QQ/\ZZ)$ via $g\mapsto t_{g}$ where 
\[t_{\sum_{i=1}^r \vec{\mathbf{m}}_i}\left(\sum_{i=1}^r \vec{\mathbf{r}}_i\right) = \sum_{i=1}^r \frac{\vec{\mathbf{m}}_i\cdot \vec{\mathbf{r}}_i}{p^{i}}\in \QQ/\ZZ .\]
Note that in particular $t_{g}(h) = t_{h}(g)$ for all $g,h\in G$.
The goal is to find $C = \langle g\rangle$ such that $\iota\notin C$ and $t_{\iota}(g)\neq 0$.
Since $\iota$ has order $p$, one can write $\iota = \sum_{i=1}^r p^{i-1}\vec{\mathbf{k}}_i$ where $\vec{\mathbf{k}}_i\in (\ZZ/p\ZZ^{n_i})$. Let $g =\sum_{i=1}^r \vec{\mathbf{m}}_i\in G$, we have 
\[t_{\iota}(g) = \sum_{i=1}^r \frac{p^{i-1}\vec{\mathbf{k}}_i\cdot \vec{\mathbf{m}}_i}{p^i}=\frac{1}{p}\sum_{i=1}^r \vec{\mathbf{k}}_i\cdot \vec{\mathbf{m}}_i. \]
In particular, we can take each $\vec{\mathbf{m}}_i\in (\ZZ/p\ZZ)^{n_i}$. More rigorously, we can use the projection $\ZZ/p^i\ZZ\rightarrow \ZZ/p\ZZ$, which yields a projection $\phi : G\rightarrow G[p] =   \left(\ZZ/p\ZZ\right)^{\sum_{i=1}^r n_i}$. We have $t_{\iota}(g) = t_{\iota}(\phi(g))$. Define $m : G\rightarrow \ZZ$ by 
\[m\left(\sum_{i=1}^r \vec{\mathbf{m}}_i\right) = \min\left(\{i\in \{1,\cdots, r\} : \vec{\mathbf{m}}_i \neq 0 \}\cup\{0\}\right),\]
it is invariant under the choice of basis. Note that we took $G$ non-cyclic,  hence $\sum_{i=1}^r n_i\geq 2$.

\begin{itemize}
	\item \emph{Case 1.} $n_{m(\iota)}\geq 2$. Take $g$ with $\vec{\mathbf{m}}_j= 0$ if $j\neq m(\iota)$, and $\vec{\mathbf{m}}_{m(\iota)}$ is some nonzero vector such that $\vec{\mathbf{k}}_{m(\iota)}\cdot \vec{\mathbf{m}}_{m(\iota)}\neq 0$ in $\ZZ/p\ZZ$, and $\vec{\mathbf{m}}_{m(\iota)}$ is not collinear to $\vec{\mathbf{k}}_{m(\iota)}$. This is easy to find, if $\vec{\mathbf{k}}_{m(\iota)}$ has two nonzero coordinates, take $\vec{\mathbf{m}}_{m(\iota)}$ with only a $1$ at those coordinates, and $0$ everywhere else. If $\vec{\mathbf{k}}_{m(\iota)}$ has only one nonzero coordinate, then take $\vec{\mathbf{m}}_{m(\iota)}$ to contain two $1$'s, one where $\vec{\mathbf{k}}_{m(\iota)}$ is supported, and one where it's not.  Since the two vectors are not collinear, we have $\iota\notin \langle g\rangle$ as desired.
	\item \emph{Case 2.} $n_{m(\iota)} = 1$ and $m(\iota)<r$. Here we repeat the same process as before, take $g$ such that $\vec{\mathbf{m}}_{m(\iota)} = (1)$. If $\iota\in \langle g\rangle$, which happens if $\vec{\mathbf{k}}_{j}=0$ for $j\neq m(\iota)$, then pick $\vec{\mathbf{m}}_{r}$ to be a vector with one $1$ and the rest $0$. We cannot have $\iota\in \langle g\rangle$ anymore because in this case $\vec{\mathbf{k}}_{r}=0$, and so to have $\iota = g^\ell$ we would need $p^{r}|\ell$, but that would make $g^\ell = 0$, which is absurd. Again we have found a suitable $g$.
	\item \emph{Case 3.} $n_{m(\iota)} = 1$ and $m(\iota)=r$. In this case, assume we found such a $g$, then we would need $\vec{\mathbf{m}}_{r}\neq 0$. Recall that we can take $\vec{\mathbf{m}}_{r} = (\ell)$ with $1\leq \ell\leq p-1$, and so $g$ has order $p^{r}$. In particular, ${p^{r-1}}g$ has order $p$, and since $n_r = 1$, $p^{r-1}$ annihilates all smaller factors of $G$, so $m({p^{r-1}}g)= r$ so ${p^{r-1}}g\in N$ and so $\langle g\rangle \cap N\neq \varnothing$. So $\iota \in \langle g\rangle$, and there are no valid choices for $g$.
\end{itemize}
We have proved: 

\begin{proposition} \label{prop:sha:abelian} Assume $G$ is abelian. If $G_p = \bigoplus_{i=1}^r \left(\ZZ/p^i\ZZ\right)^{n_i}$ with $n_r\neq 0$ and $\mathscr{C}$ is the collection of cyclic subgroups of $G$, we have $\Sh^2_\mathscr{C} (\Lambda)=0$ unless $G_p$ is not cyclic, $n_r=1$ and $N$ is equal to the $p$-torsion elements of the summand $\ZZ/p^r\ZZ$, in which case $\Sh^2_\mathscr{C} (\Lambda)\cong N$. 
\end{proposition}
 
The last condition corresponds to $n_r=1$ and $m(\iota) = r$, hence $N$ cannot be supported on any summand but  $\ZZ/p^{r}\ZZ$. This case is not common.
\begin{corollary}\label{cor:abeliannum} If $G$ is abelian then $\tau(\T) \in \{1, p\}$.
\end{corollary}

\begin{example}\label{ex:z2z2} Assume $G = (\ZZ/2\ZZ)^2$. Here $r = 1$ and $n_2 = 2\neq 1$ so by Proposition \ref{prop:sha:abelian} we get $\Sh(\mathbf{T}) = 1$. 
\end{example}
\begin{example}\label{ex:z2z4} Assume $G = \mathrm{Gal}(K/k)\cong \ZZ/2\ZZ\times \ZZ/4\ZZ$ with $\iota = (0,2)$. The only noncyclic subgroup is $H = (\ZZ/2\ZZ)^2$. As it turns out, $H^2(H, \Lambda_1)\cong H$ and $\iota\in \mathrm{Ker}(G\rightarrow H)$, hence  if $G$ does not appear itself as a decomposition group, then $\Sh^1(\mathbf{T}) = N$ and $\tau(\mathbf{T}) = 1$. 	
\end{example}
\begin{example}\label{ex:split:abelian}
In the same spirit as previous examples, an immediate application of the proposition implies the triviality of $\Sh^1(\T)$ when $G$ is abelian and (\ref{ses:group}) splits.
\end{example}

%% file: denom_rest
The computation only depends on the $p$-torsion points of the abelianization of $G$, therefore we will assume without loss of generality that $G$ is a non-cyclic $p$-group.  

Let $G'$ denote the commutator subgroup of $G$.  

Let $\alpha\in G$ such that $\alpha\notin G'$ and $\alpha^p\in G'$ (we only care about the $p$-torsion points of $G^{\mathrm{ab}}$). We fix an isomorphism $G/G'\cong \mathrm{Hom}(G, \QQ/\ZZ)$ so that there is $g\in G$ such that $t_{\alpha}\in \langle t_{g}\rangle$ and $t_{\alpha}(g) \neq 0$. Indeed, to do so, we can write a presentation of $G^\mathrm{ab}$ as $\prod_i (\ZZ/p^{n_i}\ZZ)$ where $t_{\alpha}$ is only supported on one summand, then we can take the same morphism as in previous subsection and pick $g\in G$  such that $t_{g}$ is a generator of that summand.

If $\iota\notin \langle g\rangle$ then $\hat{H}^2(\langle g\rangle, \Lambda) = \langle g\rangle$, and the image of $\alpha$ under $\hat{H}^2(G, \Lambda)\rightarrow \hat{H}^2(\langle g\rangle, \Lambda)$ is not trivial,  therefore $t_{\alpha}\notin \Sh_\mathscr{C}^2(\Lambda)$. If $\iota\in g$, then since $t_{\alpha}$ is $p$-torsion, so is its image under the restriction map $G^\mathrm{ab}\rightarrow \langle g\rangle$, and  $\langle g \rangle$ is cyclic, hence has a unique subgroup of $p$-torsion elements, hence $t_{\alpha}$ is sent into $N\subset \langle g\rangle$. By Theorem \ref{thm:h1h2} we have $\hat{H}^2(\langle g\rangle, \Lambda) = \langle g\rangle/N$, hence $t_{\alpha}$ is sent to $0$ via the restriction map. As a side note, this further explains why we only need to cover $p$-torsion elements, since the rest need not be sent into $p$-torsion elements of $\langle g\rangle$.

Again, by normality of $N$ and unicity of $p$-torsion elements in cyclic groups, we have that $\iota\in \langle g\rangle$ if and only if $\iota\in \langle \alpha\rangle$. Also by normality of $N$ in $G$, if $\iota$ belongs to the cyclic group generated by $\alpha$, it belongs to the cyclic group generated by any conjugate of $\alpha$, so the condition $\iota\in \langle \alpha\rangle$ only depends on $t_{\alpha}$ and not the choice of lift. 

Let us restate those observations in a couple lemmas. 

\begin{lemma}\label{lem:condition_sha} Assume $G_p$ is non-cyclic.  Fix $\alpha \in G$ with image $t_{\alpha}$ in $G^\mathrm{ab} \supset\Sh^2(\Lambda)$. If $t_{\alpha}\in \Sh^2(\Lambda)$, then $\alpha^p\in G'$ and $\iota\in \langle \alpha\rangle$.  
\end{lemma}
{
\begin{lemma}\label{condition_sha_sufficient} Fix $\alpha\in G$ satisfying the conditions of Lemma \ref{lem:condition_sha}. We have $t_{\alpha}\notin \Sh^2_\mathscr{C}(\Lambda)$ if and only if and there is $t_{g}\in G^{\mathrm{ab}}$ such that $t_{\alpha}(g)\neq 0$ and $\iota\notin \langle g\rangle$.
\end{lemma}
\begin{proof} This follows immediately from the discussion above and Theorem \ref{thm:h1h2}
\end{proof}

\begin{example} We have computed $\Sh^2_\mathscr{C}(\Lambda)$ in Example \ref{ex:quat:computations}. We have seen that for every $g\in G$ we have $t_{g}\in \Sh^2_\mathscr{C}(\Lambda)$. We can check that for every $g\in G$ we have $g^2\in \langle \iota\rangle = G'$ so both conditions hold. 
\end{example}

\begin{rem}
Assume $G$ is abelian, i.e. $G'=\{1\}$. The conditions of Lemma \ref{lem:condition_sha} become: $\alpha$ is $p$-torsion and $\iota\in \langle\alpha\rangle$. In particular $N\subset \langle \alpha\rangle$ and the two groups have the same size, hence the condition can be reduced to $\alpha\in N$. We therefore recover the observation of the previous subsection that $\Sh^2(\Lambda)\subset N$ in the abelian case. Therefore, if $G$ is abelian then $|\Sh^2(\Lambda)| \in \{1, p\}.$  
\end{rem}

The abelian case already shows that the conditions of Lemma \ref{lem:condition_sha} are only necessary and not sufficient. This is because in the choice of $g\in G$ such that $t_{\alpha}(g)\neq 0$ we have required so far that $t_{\alpha}\in \langle t_{g}\rangle$. We now give a non-abelian example.

\begin{example}\label{ex:d8}
Assume $p=2$ and $G$ is the dihedral group $D_4$. We write a presentation of $D_4 = \langle \alpha,\beta\rangle$ with $\alpha^4=\beta^2=1$, and $\beta\alpha\beta = \alpha^3$. There is only one choice of $N$, that is $N = Z(G)$ and $\iota= \alpha^2$. Similarly to Example \ref{ex:quat:computations}, we have $G'=\langle\iota\rangle$ and $G^\mathrm{ab}= (\ZZ/2\ZZ)^2$. We choose a presentation sending $\alpha$ to $(1, 0)$ and $\beta$ to $(0, 1)$, with the same identification with $\mathrm{Hom}(G, \QQ/\ZZ)$ as in the previous subsection.  

In this case $\iota\notin \langle\beta\rangle$ so $\beta\notin \Sh^2(\Lambda)$. Also, $\alpha$ satisfies the conditions of \ref{lem:condition_sha}, but $t_{\alpha}(\alpha\beta)=1/2$, and $\iota\notin \langle\alpha\beta\rangle$, hence $\alpha\notin \Sh^2(\Lambda)$ either. We have just shown that $\Sh^2(\Lambda) = 0$ and if $\mathrm{Gal}(K/k) = G$ then $\tau(\mathbf{T}) = 2$.

Also note that $D_4$ and $Q_8$ are the only nonabelian groups of order $8$ so we have computed all possible Tamagawa numbers for $|G|=8$.  Only two of those groups have potentially nontrivial denominator, those groups are $Q_8$ and $\ZZ/2\ZZ\times \ZZ/4\ZZ$ where $N$ is respectively $Z(Q_8)$ and $\langle (0,2)\rangle$.
\end{example}

\begin{proposition}\label{prop:denom:iotanotin}
If $\iota\notin G'$ then $\Sh^2_{\mathscr{C}}(\Lambda)$ only depends on $G^\mathrm{ab}$, and can be computed by Proposition \ref{prop:sha:abelian}. In particular, $\Sh^2(\Lambda)\subset N$.
\end{proposition}
\begin{proof}
In order to have both $\alpha^p\in G'$ and $\iota\in \langle \alpha\rangle$, we need $\alpha^p=0$, and therefore $|\langle \alpha\rangle| = p = |N|$, thus $N = \langle\alpha\rangle$ and we are left to find an element  $\beta\in G^\mathrm{ab}$ such that $t_{\alpha}(\beta)\neq 0$ as in the previous subsection. 
\end{proof}
\begin{rem}
Applying Proposition \ref{prop:denom:iotanotin} comes with one crucial caveat. In the previous section, we assume $G$ abelian and $G_p$ not cyclic. This is due to us knowing the case when $G_p$ is cyclic. However, it is possible to have $G$ not abelian with $G_p$ non cyclic, but the $p$-Sylow of $G^\mathrm{ab}$ is cyclic (and containing the nontrivial projection of $N$). In that case, the computations done in the previous section remain true and we obtain $\Sh^2_\mathscr{C}(\Lambda) = N$ (since, with notations of that section, $n_r = 1$ and $m(\iota)=r$). In order for $\Sh^2(\Lambda)$ to be trivial in that case, we need a decomposition group whose $p$-Sylow is not cyclic and whose projection on $G^\mathrm{ab}$ is onto. In general, we will see in Proposition \ref{prop:full:sha:abelian}, that if $\Sh^2_\mathscr{C}(\Lambda) = \ZZ/p\ZZ$, then $\Sh^2(\Lambda) = 0$ if and only if there is a subgroup with non-cyclic $p$-Sylow subgroups such that the projection on the summand of the $p$-Sylow subgroup of $G^\mathrm{ab}$ containing $N$ is onto.
\end{rem}

\begin{rem}Proposition \ref{prop:denom:iotanotin} covers in particular the case where (\ref{ses:group}) splits. 
\end{rem}

The only difficulty of the computation comes therefore when $N\subset G'$. As shown in Example \ref{ex:quat:computations}, we can have $|\Sh^2(\Lambda)|>|N|$.

For any $\alpha\in G$ let us write $I_\alpha = \{g\in G : t_{\alpha}(g)\neq 0\}$ and $\dot{I}_\alpha = \cap_{g\in I_\alpha} \langle g\rangle$. We can rewrite Lemmas \ref{lem:condition_sha} and \ref{condition_sha_sufficient} as: $\alpha\in \Sh^2_\mathscr{C}(\Lambda)$ if and only if $\alpha^p\in G'$ and $\iota\in \dot{I}_\alpha$. As we saw, this implies in particular,  that $\iota\in \langle \alpha\rangle$.
\begin{proposition}
Assume there is $0\neq t_{\alpha} \in \Sh^2_\mathscr{C}(\Lambda)$. Then $t_{\alpha}\in \Sh^2(\Lambda)$ if and only if $\{D\in \mathscr{D} : D_p\text{ is not cyclic, }D\cap I_\alpha \neq \varnothing \}=\varnothing$. In other words, $t_{\alpha}\notin \Sh^2(\Lambda)$ if and only if there is a ramified prime $\mathfrak{p}\in K$ such that the corresponding decomposition group $G(\mathfrak{p})$ has non-cyclic $p$-Sylow subgroup and contains some element of $I_\alpha$. 
\end{proposition}
\begin{proof}
This is simply because if $S\in \{D\in \mathscr{D} : D_p\text{ is not cyclic, }D\cap I_\alpha \neq \varnothing \}$ then by Theorem \ref{thm:h1h2} we get that the image of $\alpha$ under the restriction map  $\hat{H}^2(G, \Lambda)\rightarrow \hat{H}^2(S, \Lambda) = S^\mathrm{ab}$ is nonzero since it is supported on some $g\in I_\alpha\cap S$.
\end{proof}

\begin{example} Let $G = \ZZ/2\ZZ\times \ZZ/4\ZZ$ and $\iota = (0,2)$.
We have seen in Example \ref{ex:z2z4} that $I_\iota = \{(1,1), (0,1)\}$ and therefore $\dot{I}_\iota = N$. In order to have $\Sh^2(\Lambda) = 0$ we would need to have a non-cyclic decomposition group containing $(1,1)$ or $(0,1)$. For either one of these elements, the only non-cyclic group containing them is $G$, so $\Sh^2(\Lambda)=0$ if and only if $K/k$ contains an inert prime.
\end{example}

This argument together with Proposition \ref{prop:sha:abelian} yields a full description of the triviality of $\Sh^2(\Lambda)$ in the case where $G$ is abelian.
\begin{proposition}\label{prop:full:sha:abelian}
If $G$ is Galois and abelian, and $\Sh^2_\mathscr{C} (\Lambda) \neq 0$, then write $G_p = \prod_{k=1}^\ell \ZZ/p^{i_k}\ZZ$ with $\ell>1$ and $i_{\ell}>\max(i_1,\cdots,i_{\ell-1})$  with $N\subset \ZZ/p^{\ell}\ZZ$ according to Proposition \ref{prop:sha:abelian}. Then $\Sh^2(\Lambda) = 0$ if and only if there is a ramified prime for $K/k$ with non-cyclic decomposition group  $H$ such that the projection of $H$ on the summand $\ZZ/p^{i_\ell}\ZZ$ is onto.

This method also computes $\Sh^2(\Lambda)$ when $\iota$ does not belong to the derived subgroup $G'$ according to Proposition \ref{prop:denom:iotanotin}. 
\end{proposition}
\begin{proof}
This is a direct consequence of the fact that, in this case,  $\alpha \in I_\alpha$ if and only the projection of $\langle\alpha\rangle$  on $\ZZ/p^{i_\ell}\ZZ$ is onto. 
\end{proof}

%% file: 2groups
When $p=2$, we used Lemma \ref{condition_sha_sufficient} to compute $\Sh^2_\mathscr{C}(\Lambda)$ for $G$ with $|G_2|<256$, and for $G_2$ the first $29631$ groups of order $256$. The results are available at \cite{comput_website}. We can get the following bounds.

\begin{proposition}
Assume $|G_2|\leq 128$. Then $|\Sh^2_\mathscr{C}(\Lambda)| \in \{1, 2, 4, 8\}$. The only case with $|\Sh^2_\mathscr{C}(\Lambda)| = 8$ happens for $G = M_4(2)._{15}D_4$. In particular, for such Galois groups we have 
\[\frac{1}{4}\leq \tau(\mathbf{T})\leq 2 .\]
\end{proposition}
Here $M_4(2)._{15}D_4$ is the small group of GAP label $(128, 802)$, the $15$th non-split extension of $D_4$ by $M_4(2)$ acting via $D_4/\ZZ/2\ZZ = (\ZZ/2\ZZ)^2$.

So far over half the groups of order $256$ were checked, and no group was found that would give[] $|\Sh^2_\mathscr{C}(\Lambda)|\geq 16$.

%% file: gen_case
\subsection{Bounds for the Tamagawa number}

It is not clear whether the bounds in Proposition \ref{prop:bounds} are tight. The upper bound certainly is, and the lower bound is tight as we have seen in Example \ref{ex:quat:computations}. Other groups computed in \S \ref{subsec:2groups} reach this bound but in general $G^{\mathrm{ab}}[p]$ is much smaller than $|G|$. If $G_p$ is abelian, we know that the denominator is at most $p$, and so $1\leq \tau(\mathbf{T})\leq p$. 

It is not clear in general if we can have $N = G'$ and $\Sh^2_\mathscr{C}(\Lambda) = G/N$. The only such example currently is $G = Q_8$. For $p=2$ there is no other example of size $\leq 128$. We do not know if there is a finite number of such groups. Even more, it seems that for a fixed $p$, there might not be universal lower bound for all such tori. In particular, there might be algebraic tori in the general symplectic groups with Tamagawa numbers arbitrarily small, but if true this would require tori of very large ranks to get $|\Sh^2_\mathscr{C}(\Lambda) |> 8$.

\subsection{Non-central subgroup}

The main application of this article happens when $p=2$, in which case $N$ being normal and central are equivalent. When taking $p>2$ we have chosen to focus on $p$ central rather than the more general $p$ normal. This more restricted setting allowed us to reformulate most of results in terms of $p$-Sylow subgroups of $G$, whereas for $N$ normal, we would have needed to write everything in terms of the transfer mapping. Most results in the article remain true, including the bound of Proposition \ref{prop:bounds}. However, Proposition \ref{prop:trans_cyclicness} is not true anymore, therefore the following computations should be written in terms of the transfer mapping rather than $p$-Sylow subgroups in the vein of Lemma \ref{lem:transfer:generic} which still holds true.

\begin{example}
Assume $G = D_{9}$ and $N$ is its unique normal subgroup of order $3$. Then $G_3$ is cyclic, but $\hat{H}^1(G, \Lambda) = \ZZ/3\ZZ$. This is indeed a direct consequence of the triviality of the transfer map in that case. The assertion follows from Lemma \ref{lem:transfer:equal}
\end{example}

\subsection{Non-Galois extension}

The definition of our tori also makes sense in the setting of a non-Galois extension $K/k$. Let $K^\sharp$ denote the corresponding Galois closure. We then have the following extension diagram 

   \begin{tikzpicture}

    \node (Q1) at (0,0) {$k$};
    \node (Q2) at (0,1) {$K^+$};
    \node (Q4) at (0,3) {$K^\sharp$};
    \node (Q3) at (0,2) {$K$};

    \draw (Q1)--(Q2);
    \draw (Q2)--(Q3) node [midway, right] {$p$};
    \draw (Q1) to[out=130,in=-130] node [midway, left] {$G$} (Q4) ;
    \draw (Q3) -- (Q4) node [midway, right] {$N_2$};
    \draw (0.5,1) to[out=40,in=-40] node [midway, right] {$N_1$} (0.5,3) ;
    \end{tikzpicture}

where $G = \mathrm{Gal}(K^\sharp/k)$, $N_1 = \mathrm{Gal}(K^\sharp/K^+)$, and $N_2 = \mathrm{Gal}(K^\sharp/K)$. Note that if $p=2$, we have that $K/K^+$ is necessarily Galois, and so $N_2$ is normal in $N_1$. 

We recall the definition of our torus of rank $[K^+:k](p-1)+1$:
\begin{equation}\label{eq:torus:nongal}\mathbf{T} = \mathrm{Ker}\left( \mathbb{G}_m\times_{\mathrm{Spec}(k)}\mathbf{R}_{K/k}(\mathbb{G}_m)\underset{(x, y)\mapsto x^{-1}N_{K/K^+}(y)}{\longrightarrow} \mathbf{R}_{K^+/k}(\mathbb{G}_m)\right).\end{equation}

Observe that $\mathbf{X}^\star(\res_{K^+/k}\mathbb{G}_m) = \ZZ[G/N_1]$, and $\mathbf{X}^\star(\res_{K/k}\mathbb{G}_m) = \ZZ[G/N_2]$, and therefore the norm map $N_{K/K^+} : \res_{K/k}\mathbb{G}_m\rightarrow \res_{K^+/k}\mathbb{G}_m$ yields the corresponding map on the character lattices $\ZZ[G/N_1]\rightarrow \ZZ[G/N_2]$ defined by \[gN_1\mapsto \sum_{hN_2\in N_1/N_2}ghN_2 = g\cdot \left(\sum_{h\in N_1/N_2}hN_2\right).\]

By definition of $\mathbf{T}$ as kernel, we get the definition of $\mathbf{X}^\star(\mathbf{T})$ through the dual sequence:
\[0\rightarrow \ZZ[G/N_1]\rightarrow \ZZ\oplus \ZZ[G/N_2] \rightarrow\mathbf{X}^\star(\T) \rightarrow 0,\]
where the first map is the direct sum of the augmentation map and the norm map above. This corresponds with the description made in \cite{cortella}.

When $K = K^\sharp$ is Galois, then $N_2 = \{1\}$, and hence we get $N_1 = N$ as in the previous sections.

\begin{rem}
We now write a description of $\mathbf{X}^\star(\T)$ which is not particularly interesting for theoretical purposes, but does simplify its construction in SAGE. 

Let $\phi : \ZZ[G/N_1]\rightarrow \ZZ[G/N_2]$ be the map defined above. Its image is the span (as $G$-module) of $\phi(N_1)$. 
One can check that $\phi|_{\ZZ[G/N_1]^{N_1}} :  \ZZ[G/N_1]^{N_1} \rightarrow \ZZ[G/N_2]^{N_1}$ is onto and contains $\phi(N_1)$. Consequently, the image of $\phi$  can be written $\overline{\ZZ[G/N_2]^{N_1}}$, which we use to denote  $G\cdot\ZZ[G/N_2]^{N_1}$. If $K^+/k$ is Galois then $N_1$ is normal in $G$, hence $\ZZ[G/N_1]^{N_1} = \ZZ[G/N_1]$ and $\mathrm{Im}(\phi) = \mathrm{Im}(\phi|_{\ZZ[G/N_1]})=\ZZ[G/N_2]^{N_1}$.

Using this description, we can write 
\[\mathbf{X}^\star(\mathbf{T}) = \ZZ[G/N_2] / J\overline{\ZZ[G/N_2]^{N_1}},\]
where $J$ denotes the augmentation ideal. This explains why we build the $N_1$-invariant elements of $\ZZ[G/N_2]$ and then complete the result into a $G$-submodule in the following examples. Building the map $\phi$ and its image is also possible, but more cumbersome.
\end{rem}
\subsection{Non-Galois extensions of degree 4}

Assume $K$ is a non-Galois field extension of degree $4$; in this setting $K^+/k$ is quadratic hence Galois. Looking at the transitive subgroups of $S_4$, we observe that the only possibility of having an intermediate field $K^+/k$ of degree $2$ is to have $G$ be a group of order $4$ or $D_4$. Therefore since $K$ is non-Galois, we get $G = D_4$. Taking $N_2$ to be any non-normal subgroup of $D_4$ of  degree $2$ and $N_1$ any (normal) subgroup of $D_4$ of order $4$ containing $N_2$. Then SAGE computations show that $|\hat{H}^1(G, \mathbf{X}^\star(\mathbf{T}))| = 2$ and $|\Sh^2|\leq |\Sh^2_\mathscr{C}(\mathbf{X}^\star(\mathbf{T}))| = 1$.

\adjustbox{scale = 0.7, center}{
\begin{lstlisting}
sage: G = DihedralGroup(4)
sage: N2 = [h for h in G.subgroups() if h.order() == 2][1]
sage: N1 = [h for h in G.subgroups() if h.order() == 4][0]
sage: N2.is_normal(G)
False
sage: N2.is_normal(N1)
True
sage: Z = GLattice(N2, 1)
sage: IL = Z.induced_lattice(G)
sage: SM = IL.fixed_sublattice(N1)
sage: SL = SM.complete_submodule().zero_sum_sublattice()
sage: L = IL.quotient_lattice(SL)
sage: L.Tate_Cohomology(1)
[]
sage: L.Tate_Shafarevich_lattice(2)
[]
\end{lstlisting}}

\begin{proposition}\label{prop:deg4} Let $K/k$ be an extension of degree $4$ with intermediate extension $K^+/K$ of degree $2$. Define $\mathbf{T}$ as in (\ref{eq:torus:nongal}). If $K/k$ is Galois and $\mathrm{Gal}(K/k) = (\ZZ/2\ZZ)^2$ then we get $\tau(\T) = 2$, else $\tau(\mathbf{T}) = 1$.
\end{proposition}
\begin{proof} The only non-Galois case is computed above. The numerator in the case $K/k$ Galois follows from Theorem \ref{thm:numerator} for $G = (\ZZ/2\ZZ)^2$ and Proposition \ref{prop:cyclic:h1} for $G = \ZZ/4\ZZ$. The denominator in the Galois case is always trivial, when $G = \ZZ/4\ZZ$ this follows from Proposition \ref{prop:cyclic:denom} and when $G = (\ZZ/2\ZZ)$ it was done in Example \ref{ex:z2z2}.
\end{proof}

\begin{example}
Motivated by the application of the Tamagawa number in the formula of \cite{abvarcount}, we decided to compute the Tamagawa number related to the centralizer of the Frobenius of either Hyperelliptic curves with isomorphic Jacobians appearing in \cite{howe}. We can consider the hyperelliptic curve $X$ over $\mathbb{F}_3$ defined by the equation $y^2 = x^5+x^3+x^2-x-1$. The Frobenius polynomial of this curve is $x^4 - x^3 + x^2 - 3x + 9$ which yields a non-Galois extension of $\QQ$. Therefore if $\T$ is the centralizer of this Frobenius in $\mathrm{GSp}_4$, we get $\tau(\T) = 1$. 
\end{example}
\subsection{The numerator of the Tamagawa number for general non-Galois extensions.}

As before, we define the auxilliary torus 

\[\T_1 =  \mathrm{Ker}\left( \mathbf{R}_{K/\mathbb{Q}}(\mathbb{G}_m)\underset{N_{K/K^+}}{\longrightarrow} \mathbf{R}_{K^+/k}(\mathbb{G}_m)\right)= \res_{K^+/\QQ}\res_{K/K^+}^{(1)}(\mathbb{G}_m).\]

We have the exact sequence 
\begin{equation}\label{ses:tori2}
1\rightarrow \T_1\rightarrow \mathbf{T} \rightarrow \mathbb{G}_m\rightarrow 1.
\end{equation}

In particular, 
\[0\rightarrow \ZZ\rightarrow \mathbf{X}^\star(\T)\rightarrow \mathbf{X}^\star(\T_1)\rightarrow 0,\]
and by Hilbert 90, we get $|\hat{H}^1(G, \mathbf{X}^\star(\T))|\leq |\hat{H}^1(G, \mathbf{X}^\star(\T_1))|$.

Let $\Lambda = \mathbf{X}^\star(\T)$ and $\Lambda_1 = \mathbf{X}^\star(\T_1)$. From the same reasoning as for $\T$, we get the exact sequence
\[0\rightarrow \ZZ[G/N_1]\rightarrow \ZZ[G/N_2]\rightarrow \Lambda_1\rightarrow 0.\]

The corresponding cohomology yields
\[0=\hat{H}^1(N_2, \ZZ)\rightarrow \hat{H}^1(G, \Lambda_1)\rightarrow \hat{H}^2(N_1, \ZZ)=N_1^{\mathrm{ab}}\rightarrow \hat{H}^2(N_2, \ZZ)=N_2^{\mathrm{ab}}.\]

The map $N_1^{\mathrm{ab}}\rightarrow N_2^{\mathrm{ab}}$ is just the restriction map on the dual groups $\varphi:\mathrm{Hom}(N_1, \QQ/\ZZ)\rightarrow \mathrm{Hom}(N_2, \QQ/\ZZ)$.

If $K/K^+$ is Galois, then $N_2$ is normal in $N_1$, therefore if $\alpha\in \mathrm{Ker}(\varphi)$, then $\alpha|_{N_2}=0$ and hence $\alpha$ is induced from a character of $N_1/N_2= \ZZ/p\ZZ$. In this case, $\hat{H}^1(G, \Lambda_1) = \mathrm{Ker}(\varphi) = (\ZZ/p\ZZ)^{\mathrm{ab}} = p$. In the case $K/K^+$ is not Galois, and hence $N_2$ is not normal, if $\alpha\in \mathrm{Ker}(\varphi)$ then $\mathrm{Ker}(\alpha)$ is a normal subgroup of $N_1$ containing $N_2$, but $N_2$ has prime index, hence $\alpha = 0$. Therefore we proved:

\begin{proposition} \label{prop:h1der:nongal} We have $\tau(\mathbf{T})\leq |\hat{H}^1(G, \Lambda_1)|$, and $\hat{H}^1(G, \Lambda_1)=\ZZ/p\ZZ$ if $K/K^+$ is Galois, else $\hat{H}^1(G, \Lambda_1)=0$.
\end{proposition}
\begin{proof} We have shown the second assertion. The first one comes from the formula of Theorem \ref{thm:ono_formula} and the cohomology of character lattices of (\ref{ses:tori2}).
\end{proof}

\begin{rem} If $K^+$ contains a $p$-root of unity, then $K/K^+$ is always Galois since it has degree $p$, in particular when $p=2$ any quadratic extension is Galois so in this case we always get $\tau(\T)\leq 2$. Note that from the classical results, if $p$ is the smallest prime dividing $|N_1| = [K^\sharp:K^+]$ then $N_2$ is automatically normal in $N_1$. In particular this holds when $N_1$ is a $p$-group.  
\end{rem}

The cohomology of the sequence of character lattices associated to (\ref{ses:tori2}) yields 
\[0\rightarrow \hat{H}^1(G, \Lambda)\rightarrow \hat{H}^1(G, \Lambda_1)\rightarrow \hat{H}^2(G, \ZZ).\]
Now using the fact that $\hat{H}^1(G, \Lambda_1) = \mathrm{Ker}(\mathrm{Hom}(N_1, \QQ/\ZZ)\rightarrow \mathrm{Hom}(N_2, \QQ/\ZZ))$, we get:
\begin{proposition}  We have the equality
\begin{align*}\hat{H}^1(G, \Lambda) &= \mathrm{Ker}\left(\hat{H}^1(G, \Lambda_1)\rightarrow \hat{H}^2(G, \ZZ) \right)\\
&=\{\alpha\in \mathrm{Hom}(N_1, \QQ/\ZZ) : \alpha|_{N_2}=0, \quad \alpha\circ \mathrm{tr}^G_{N_1}= 0\}. \end{align*}
\end{proposition}
The last equality of the Proposition seems a little ambiguous since $\mathrm{tr}^G_{N_1}$ is valued in $N_1^{\mathrm{ab}}$, but note that elements of $\mathrm{Hom}(N_1, \QQ/\ZZ)$ factor through $N_1^{\mathrm{ab}}$ and hence we only need to look at pre-transfers. 

\subsection{The case of CM-fields}\label{sec:cm:nongal}

Assume $k= \QQ$ and let $K$ be a CM-field, with maximal totally real subfield $K^+$. Let $\rho$ denote the involution given by complex conjugation, it is central in $G$. Since $K^+$ is totally real, we also have $\rho \in N_1$, but $K/K^+$ is imaginary so $\rho\notin N_2$.

\begin{lemma}\label{lem:num:transfer}
We have $\hat{H}^1(G, \Lambda) \subset \ZZ/2\ZZ$, and $\hat{H}^1(G, \Lambda) = \ZZ/2\ZZ$ if and only if $\mathrm{tr}^G_{N_1}(G)\subset N_2$.
\end{lemma}
\begin{proof}
The first assertion follows from the fact 
\[\hat{H}^1(G, \Lambda)\subset \hat{H}^1(G, \Lambda_1) = \mathrm{Hom}(N_1/N_2, \QQ/\ZZ)= \ZZ/2\ZZ.\]
Here $\mathrm{Hom}(N_1/N_2, \QQ/\ZZ)$ is identified with characters of $N_1$ trivial on $N_2$. The only nontrivial element of this group is the character $\alpha : N_1\rightarrow \QQ/\ZZ$ defined by $\alpha(N_2) = 0$ and $\alpha (\rho N_2) = 1/2$.

Therefore, by the description of $\hat{H}^1(G, \Lambda)$ in the previous section, we only need to determine if the image of $\mathrm{tr}^G_{N_1}$ is contained in $\mathrm{Ker}(\alpha) = N_2$. If this is the case then $\hat{H}^1(G, \Lambda) = \ZZ/2\ZZ$, else $\hat{H}^1(G, \Lambda) = 0$.   
\end{proof}

\begin{corollary} If $[K^+:\QQ]$ is odd, then $\hat{H}^1(G, \Lambda) = 0$. 
\end{corollary}
\begin{proof}
Assume $[K^+:\QQ]$ is odd. Since $\rho$ is central, we have $\mathrm{tr}^G_{N_1}(\rho) = \rho^{|G/N_1|}= \rho^{([K^{+}:\QQ])}= \rho\notin N_2$ since $K$ is not totally real.
\end{proof}

\begin{corollary}
If $K$ is a CM-field of degree $6$ then $\hat{H}^1(G, \Lambda)=1$. 
\end{corollary} 

We now know the numerators for all extensions of degree $4$ and $6$. 

\begin{theorem}\label{thm:h1nongal}
Let $K/\QQ$ be a CM-field and $\Lambda = \mathbf{X}^\star(\T)$. We have $\hat{H}^1(G, \Lambda)\subset \ZZ/2\ZZ$. Moreover, $\hat{H}^1(G, \Lambda) = 0$ if and only if there is $g\in G$ such that $|\langle g\rangle \backslash G/N_2|$ is odd, where $G= \mathrm{Gal}(K^\sharp/\QQ)$ and $N_2 = \mathrm{Gal}(K^\sharp/K)$. 

In particular, if $K/\QQ$ is Galois, then we recover the result that $\hat{H}^1(G, \Lambda) = 0$ if and only if the $2$-Sylow subgroups of $\mathrm{Gal}(K/\QQ) \cong G/N_2$ are cyclic.
\end{theorem}
\begin{proof}
We will use \ref{lem:num:transfer} and compute the transfer map $\mathrm{tr}^G_{N_1}$. This proof will be somewhat technical so the reader can refer to a simplified example going through the main ideas of the proof in Example \ref{ex:postsigma}.

Decompose $G$ into the left cosets: $G = \bigsqcup_{i=1}^{n} x_i N_1$ and write $X = \{x_1, \cdots, x_n\}$. Fix $g\in G$. Let $\tilde{\sigma}\in S_X$ such that $g x N_1 = \tilde{\sigma}(x)N_1$ for all $x\in X$. Let $\tilde{\sigma} = \tilde{\zeta}_1\cdots \tilde{\zeta}_\ell$ be the decomposition of $\tilde{\sigma}$ as product of disjoint cyclic permutations.

For all $1\leq i\leq \ell$ write $\tilde{\zeta}_i = (x_{i_1}\ \cdots\ x_{i_{r_i}})$. Write 
\[g^{\tilde{\zeta}_i} = x_{i_1}^{-1}g^rx_{i_1} = (x_{i_1}^{-1}gx_{i_2})(x_{i_2}^{-1}g x_{i_3})\cdots (x_{i_{r_i}}^{-1}g x_{i_1})\in N_1. \]
Then $\mathrm{tr}^G_{N_1}(g) = \prod_{i=1}^\ell g^{\tilde{\zeta}_i}N_1'$ where $N_1'$ denotes the commutator subgroup of $N_1$. Note that this cycle decomposition is equivalent to a decomposition of $G$ into double cosets $\langle g\rangle \backslash G/N_1$ via $G = \bigsqcup_{j=1}^{\ell} \bigsqcup_{i=1}^{r_j} g^i x_{j_1} N_1 = \bigsqcup_{j=1}^{\ell}\langle g \rangle x_{j_1}N_1$. In particular, $\ell = |\langle g \rangle \backslash G/N_1|$.

We have that $N_1 = N_2\sqcup \rho N_2$, so $G = \left(\bigsqcup_{i=1}^{n}x_iN_2\right)\sqcup \left(\bigsqcup_{i=1}^n \rho x_i N_2\right).$ Let $\rho X = \{\rho x : x\in X\}$, and let $\sigma\in S_{X\sqcup \rho X}$ such that $gxN_2 = \sigma(x) N_2$ for all $x\in X\sqcup \rho X$. Since for all $x\in X$ we have $gxN_1 = \sigma(x) N_1$, then either $gxN_2 = \tilde{\sigma}(x)N_2$ or $gxN_2 = \rho\tilde{\sigma}(x)N_2$, hence for all $x\in X$ we have $\sigma(x) = \tilde{\sigma}(x)$ or $\sigma(x) = \rho\tilde{\sigma}(x)$. Decompose $\sigma = \zeta_1\cdots \zeta_t$ into disjoint cycles, where $t = |\langle g\rangle \backslash G /N_2|$. Define the action of $\rho$ on $S_{X\cup \rho X}$ via $\rho\cdot \tau(x) = \tau(\rho x)$. 

Clearly for all $1\leq i\leq t$, both $\zeta_i$ and $\rho\cdot \zeta_i$ appear in the decomposition of $\sigma$ since the latter is $\rho$-invariant.

Let $X'\subset X$ denote a set of representatives of $\langle g \rangle \backslash G/N_1$. For all $1\leq j\leq \ell$ we have the identification $\tilde{\zeta}_j = \tilde{\zeta}_x$ for some $x\in X'$, where $\tilde{\zeta}_x$ denotes the action of $\tilde{\sigma}$ on the orbit of $x$. Likewise for all $1\leq i\leq t$ we have $\zeta_i= \zeta_{x}$ or $\zeta_i = \zeta_{\rho x} = \rho\cdot \zeta_x$ for some $x\in X'$. Let $x\in X'$, the cycle $\tilde{\zeta}_x$ corresponds to two permutations $\zeta_x$ and $\zeta_{\rho x}$. If $\zeta_x \neq \zeta_{\rho x}$ then $\tilde{\zeta}_x$ has the same length as $\zeta_x$ and $\rho$ is central hence $g^{\tilde{\zeta}_x} = g^{\zeta_{x}} = g^{\zeta_{\rho x}}$, thus $g^{\tilde{\zeta}_x}\in N_2$. If $\zeta_x = \zeta_{\rho x}$, then the length of $\zeta_x$ is twice the size of $\tilde{\zeta}_{x}$, and in particular $g^{\zeta_x} = \left(g^{\tilde{\zeta}_x}\right)^2$ and $g^{\zeta_x}\in \rho N_2$. Note that
\[\eta := \left|\left\lbrace x\in X' : g^{\tilde{\zeta}_x}\notin N_2 \right\rbrace\right| = \left|\left\lbrace x\in X' : \zeta_x = \rho\cdot \zeta_{x}\right\rbrace\right|=\left|\left\lbrace 1\leq i\leq t : \zeta_i = \rho\cdot\zeta_i \right\rbrace\right|.\]
On the one hand $\mathrm{tr}^G_{N_1}(g) = \prod_{i=1}^\ell g^{\tilde{\zeta}_x}\in \rho^{\eta}N_2$ belongs to  $N_2$ if and only if $\eta$ is even. On the other hand, $\eta$ is the number of fixed points of the action of the involution $\rho$ on $\{\zeta_1, \cdots, \zeta_t\}$, hence  $\eta \equiv t\ \text{(mod $2$)}$. We got $\mathrm{tr}^G_{N_1}(g)\in N_2$ if and only if $t$ is even, and $t = |\langle g\rangle \backslash G/N_2|$. We conclude that $\mathrm{Im}(\mathrm{tr}^G_{N_1})\subset N_2$ if and only if $|\langle g\rangle \backslash G/N_2|$ is even for all $g\in G$ and then we can use Lemma \ref{lem:num:transfer}.

If $K/\QQ$ is Galois, then $N_2$ is normal in $G$, so $\langle g\rangle\backslash G/N_2$ is even for all $g\in G$ if and only if every cyclic subgroup of $\mathrm{Gal}(K/\QQ) \cong G/N_2$ has even index, which is true if and only if the $2$-Sylow subgroups of $G/N_2$ are not cyclic. 
\end{proof}

\begin{example}\label{ex:postsigma} Using the notations of the previous theorem, we will give a minimalist example to explain the reasonning. 

Assume $|G/N_1|=3$ and we pick the left coset representatives $x,y,z\in G$. Then we have 
\[G = \bigsqcup_{i=0}^1 \rho^ixN_2\sqcup \rho^iyN_2\sqcup \rho^izN_2.\]
Pick $g\in G$ with associated permutations $\tilde{\sigma}\in S_{G/N_1}$ and $\sigma \in S_{G/N_2}$ as in the theorem. We will also assume that $\tilde{\sigma} = \tilde{\zeta}_1 = (x\ y\ z)$, and $\sigma$ sends $x$ to $\rho x$, and $y$ to $\rho y$. Since $\tilde{\sigma}(z) = x$, there are then two possibilities for the image of $z$ under $\sigma$. 
\begin{itemize}
	\item $\sigma(z) = x$.  We can therefore write $\sigma = (x\ \rho y\ z) (\rho x\ y\ \rho z) = \zeta_x \zeta_{\rho x}$, and $\rho$ permutes both (disjoint) cycles. Equivalently $|\langle g\rangle \backslash G/N_2| = 2$. By centrality of $\rho$ we have $g^{\tilde{\zeta}_1} = g^{\zeta_x} = g^{\zeta_{\rho x}}\in N_2$.
	\item $\sigma(z) = \rho x$. Now $\sigma = (x\ \rho y\ z\ \rho x\ y\ \rho z)= \zeta_x.$ This is the case when $\zeta_x$ is fixed under the action of $\rho$.  Equivalently $|\langle g\rangle \backslash G/N_2| = 1$. In particular, we have $g^{\tilde{\zeta}_1}\in \rho N_2$, and $g^{\zeta_x} = (g^{\tilde{\zeta}_1})^2\in N_2$.  
\end{itemize}
This is the argument we use in the proof, the image of the pre-transfer will belong in $\rho^i N_2$ where $i$ is the number of fixed cycles in the decomposition of $\sigma$. The non-fixed cycles come in pairs, so this number has the same parity as $|\langle g\rangle \backslash G/N_2|.$
\end{example}

\subsection{Examples computed by SAGE.}

We have determined the Tamagawa numbers of the tori corresponding to CM-fields of degree $4$, and we have seen the numerator of Ono's formula is always trivial for CM-fields of degree $6$. Moreover, using SAGE, one can check that $\Sh_\mathscr{C}(\T)$ is trivial in the latter case. We get the following proposition.

\begin{proposition}\label{prop:deg6} Let $K/\QQ$ be a $CM$-field of degree $6$, then $\tau(\mathbf{T}) = 1$. 
\end{proposition}

For degree $8$ extensions, one can easily compute the numerator and $\Sh^2_\mathscr{C}(\T)$; the results are available at \cite{comput_website}. We can notice that there are only $6$ different groups $\mathrm{Gal}(K^\sharp/\QQ)$ such that $\Sh^2_\mathscr{C}(\T) \neq 0$, in all those cases we have $\Sh^2_\mathscr{C}(\T) = \ZZ/2\ZZ$.

\begin{example}\label{ex:random_nongal8} Assume $K$ is the CM-field defined by the polynomial $x^{8} + 8 x^{6} + 17 x^{4} + 9 x^{2} + 1 $. The corresponding Galois group is $C_2^3:S_4$, of order $192$ with action 8T39 on the roots. We get $\tau(\T) = 2$.  The totally real subfield is defined by $x^{4} - x^{3} - 5 x^{2} + 4 x + 3 $ and has Galois group $S_4$. 
\end{example}

\begin{example}\label{ex:1galois:3ext} Let $f_1, f_2, f_3$ the three following polynomials:
\begin{align*}
f_1 &= x^8+14x^6+36x^4+28x^2+4,\\
f_2 &= x^{8} + 28 x^{6} + 250 x^{4} + 868 x^{2} + 961,\\
f_3 &= x^{8} + 14 x^{6} + 39 x^{4} + 32 x^{2} + 8.
\end{align*}
Let $K_1, K_2, K_3$ be the corresponding CM-fields with respective tori $\T_1, \T_2,$ and $\T_3$. All have Galois group $\mathrm{Gal}(K_i^\sharp/\QQ)= C_2^2.D_4$ with label $(32, 6)$, however $\tau(\T_1)=2$ and $\tau(\T_2) = \tau(\T_3) = 1$. Indeed, despite the extensions having the same Galois group, they correspond to a different transitive action on  $8$ points, respectively 8T19, 8T20, 8T21. Here both $K_2^+$ and $K_3^+$ are Galois, with respective Galois groups  $\ZZ/4\ZZ$ and $(\ZZ/2\ZZ)^2$. The field $K_1^+$ is not Galois, the Galois group of its Galois closure is $D_4$. 
\end{example}

\begin{example}\label{ex:num1den2}  Let $K$ be the non-Galois CM-field defined by 
\[f = x^{8} - 3 x^{7} + 27 x^{6} - 85 x^{5} + 331 x^{4} - 690 x^{3} + 1513 x^{2} - 1694 x + 1801.\]
The Galois group of its closure is the Semidihedral group $QD_{16} = Q_8\rtimes \ZZ/2\ZZ$. Every decomposition group here is cyclic, hence using the table for extensions of degree 8 in \cite{comput_website} we get that $\tau(\T) = \frac{1}{2}$. This example is particularly interesting because in the Galois case, Proposition \ref{prop:cyclic:denom} tells us that if the numerator of Ono's formula equals $1$, then so does the denominator. However in this non-Galois case, we can have trivial $\hat{H}^1(\QQ, \mathbf{X}^\star(\T))$ but not $\Sh^1(\T)$. 
\end{example}
\subsection{Tori split over a CM-\'Etale algebra}

Let $K = \bigoplus_{i=1}^r K_i$ be an \'etale algebra over $\QQ$ with the intermediate \'etale subalgebra $K^+ = \bigoplus_{j=1}^m K^+_j$. We can again define the tori

\[ \mathbf{T}^K = \mathrm{Ker}\left( \mathbb{G}_m\times_{\mathrm{Spec}(\QQ)}\mathbf{R}_{K/\QQ}(\mathbb{G}_m)\underset{(x, y)\mapsto x^{-1}N_{K/K^+}(y)}{\longrightarrow} \mathbf{R}_{K^+/\QQ}(\mathbb{G}_m)\right), \]
and 
\[\T_1^K =  \mathrm{Ker}\left( \mathbf{R}_{K/\mathbb{Q}}(\mathbb{G}_m)\underset{N_{K/K^+}}{\longrightarrow} \mathbf{R}_{K^+/\mathbb{Q}}(\mathbb{G}_m)\right)= \res_{K^+/\QQ}\res_{K/K^+}^{(1)}(\mathbb{G}_m).\]

Assume that $K$ is a CM-\'etale algebra, with complex involution $\rho$ and $K^+$ is the \'etale algebra fixed by $\rho$. Since each $K_i^+$ is totally real and $\rho$ is the complex conjugation, we have $r=m$, and for each $1\leq i\leq r$ we can define  

\[ \mathbf{T}^{K_i} = \mathrm{Ker}\left( \mathbb{G}_m\times_{\mathrm{Spec}(\QQ)}\mathbf{R}_{K_i/\QQ}(\mathbb{G}_m)\underset{(x, y)\mapsto x^{-1}N_{K_i/K_i^+}(y)}{\longrightarrow} \mathbf{R}_{K_i^+/\QQ}(\mathbb{G}_m)\right), \]
and 
\[\T_1^{K_i} =  \mathrm{Ker}\left( \mathbf{R}_{K_i/\mathbb{Q}}(\mathbb{G}_m)\underset{N_{K_i/K_i^+}}{\longrightarrow} \mathbf{R}_{K_i^+/\mathbb{Q}}(\mathbb{G}_m)\right)= \res_{K_i^+/\QQ}\res_{K_i/K_i^+}^{(1)}(\mathbb{G}_m).\]

We have $\T_1^K = \prod_{i=1}^r \T_1^{K_i}$, and $\T_1^K\subset \T^K \subset \prod_{i=1}^r\T^{K_i}$ with corresponding exact sequences 
\begin{equation} 1\rightarrow \T_1^K \rightarrow \T^K \rightarrow \mathbb{G}_m\rightarrow 1,
\end{equation}
and
\begin{equation}1 \rightarrow \T^K \rightarrow  \prod_{i=1}^r\T^{K_i}\rightarrow \mathbb{G}_m^{r-1}\rightarrow 1.
\end{equation}

The respective exact sequences of character lattices are 

\begin{equation}\label{eq:etale:t0} 0\rightarrow \ZZ \rightarrow \mathbf{X}^\star(\T^K) \rightarrow \mathbf{X}^\star(\T_1^K)\rightarrow 0,
\end{equation}
and
\begin{equation}\label{eq:etale:prod}1 \rightarrow \ZZ^{r-1} \rightarrow  \bigoplus_{i=1}^r\mathbf{X}^\star(\T^{K_i})\rightarrow \mathbf{X}^\star(\T^K)\rightarrow 1.
\end{equation}

Thanks to the previous section we have computed \[\hat{H}^1(\QQ, \mathbf{X}^\star(\T_1^K))= \bigoplus_{i=1}^r \hat{H}^1(\QQ, \mathbf{X}^\star(\T_1^{K_i})) = (\ZZ/2\ZZ)^r\] and Theorem \ref{thm:h1nongal} gives us a formula for $\bigoplus_{i=1}^r\hat{H}^1(\QQ, \mathbf{X}^\star(\T^{K_i})).$

The cohomology of (\ref{eq:etale:t0}) and (\ref{eq:etale:prod}) gives us an estimate of the numerator of  $\tau(\mathbf{T})$.

We have the inequality
\[ \prod_{i=1}^r|\hat{H}^1(\QQ, \mathbf{X}^\star(\T^{K_i}))|\leq |\hat{H}^1(\QQ, \mathbf{X}^\star(\T^K))|\leq  2^r.\]

Let $K^\sharp$ be a Galois extension of $\QQ$ splitting every $\T^{K_i}$, and hence splitting $\T^K$ too. Let $G = \mathrm{Gal}(K^\sharp/\QQ)$. For each $i\in \{1, \cdots, r\}$ we let $K_i^\sharp$ denote the Galois closure of $K_i$, and $K_i^+$ its maximal totally real subfield. We write $G_i = \mathrm{Gal}(K_i^\sharp/\QQ)$, $H_i = \mathrm{Gal}(K^\sharp/K_i^\sharp)$ and $N_i = \mathrm{Gal}(K_i^\sharp/K_i)$.

Continuing with the cohomology of (\ref{eq:etale:prod}), we get 
\[0\rightarrow \hat{H}^1(G, \mathbf{X}^\star(\T^K))\rightarrow \hat{H}^1(G, \mathbf{X}^\star(\T^K_0)) = \underset{= (\ZZ/2\ZZ)^r}{\underbrace{\bigoplus_{i=1}^r \hat{H}^1(G, \mathbf{X}^\star(\T_1^{K_i}))}}\rightarrow \hat{H}^2(G, \ZZ).\]
We have computed the rightmost arrow in the previous section for a single CM-field, this time we have to consider the sum of those mappings. If $g\in G$ and $1\leq i\leq r$, we let 
\[\nu_g^i = \frac{\left|\langle gH_i \rangle \backslash (G/H_i)/N_i\right|}{2} = \frac{\left|\langle gH_i \rangle \backslash G_i/N_i\right|}{2}\in \QQ/\ZZ.\]
Let $\mathbf{k} = (k_1,\cdots, k_r)\in (\ZZ/2\ZZ)^r$, and define $\varphi_{\mathbf{i}}\in \mathrm{Hom}(G, \QQ/\ZZ)$ by $\varphi_{\mathbf{k}}(g) := \sum_{i=1}^r k_i\nu^i_{g}.$ To sum up the results of Theorem \ref{thm:h1nongal}, we get that 
\begin{lemma}\[\hat{H}^1(G, \mathbf{X}^\star(\T)) = \mathrm{Ker}(\varphi),\text{ where }\varphi : \left\lbrace\begin{array}{l}(\ZZ/2\ZZ)^r\rightarrow \mathrm{Hom}(G, \QQ/\ZZ)\\ \mathbf{k}\mapsto \varphi_{\mathbf{k}}\end{array}\right..\] 
\end{lemma}

We can recover previous computations of $\hat{H}^1(G, \mathbf{X}^\star(\T^{K_i})) = \mathrm{Ker}(\varphi \circ \eta_i)$ where $\eta_i : \ZZ/2\ZZ\rightarrow (\ZZ/2\ZZ)^r$ is the injection in the $i$th coordinate.  In particular, we again have \[\hat{H}^1(G, \mathbf{X}^\star(\prod_{i=1}^r\T^{K_i}))\bigoplus_{i=1}^r \mathrm{Ker}(\varphi\circ \eta_i) \subset \mathrm{Ker}(\varphi) = \hat{H}^1(G, \mathbf{X}^\star(\T)).\]

\begin{corollary}\label{cor:pre_split_num}
Assume $K = \bigoplus_{i=1}^r (K_i)^{\oplus j_i}$ for some $j_1, \cdots, j_r\in \mathbb{N}$. Define $I = \{1\leq i\leq r : \hat{H}^1(G, \mathbf{X}^\star(\mathbf{T}^{K_i})) = 0\}$. Then we get 
\[ \bigoplus_{i\notin I} (\ZZ/2\ZZ)^{j_i}\oplus \bigoplus_{i\in I}(\ZZ/2\ZZ)^{j_i-1}\leq \hat{H}^1(G, \mathbf{X}^\star(\T)) \leq \bigoplus_{i=1}^r(\ZZ/2\ZZ)^{j_i}.\] 
\end{corollary}
\begin{proof} For each $i$, let $\tilde{\varphi}^i$ denote the restriction of $\varphi$ to the labels corresponding to $K_i^{j_i}$. 

Let $1\leq i\leq r$.  If $\hat{H}^1(G, \mathbf{X}^\star(\mathbf{T}^{K_i})) = \ZZ/2\ZZ$, then $\tilde{\phi}^i$ is identically zero, hence $(\ZZ/2\ZZ)^{j_i}=\mathrm{Ker}(\tilde{\varphi}^i)$. If however $\hat{H}^1(G, \mathbf{X}^\star(\mathbf{T}^{K_i})) = 0$, then $\tilde{\varphi}^i$ will send every zero-sum vector  of $(\ZZ/2\ZZ)^{j_i}$  to $0$, hence $(\ZZ/2\ZZ)^{j_i-1} = \mathrm{Ker}(\tilde{\varphi}^i)$.
\end{proof}
\begin{corollary}\label{cor:split:num}
Assume $K = \bigoplus_{i=1}^r (K_i)^{\oplus j_i}$ for some $j_1, \cdots, j_r\in \mathbb{N}$ such that $\mathrm{Gal}(K^\sharp/\QQ) = \prod_{i=1}^r\mathrm{Gal}(K_i^\sharp/\QQ)$. We get 
\[\hat{H}^1(G, \mathbf{X}^\star(\T)) = \bigoplus_{i=1}^r (\ZZ/2\ZZ)^{j_i-1}\hat{H}^1(G_i, \mathbf{X}^\star(\mathbf{T}^{K_i})).\] 
\end{corollary}
\begin{proof}
In that case, the triviality of each $\varphi_\mathbf{i}$ can be tested independently on each quotient $G/H_i = G_i$. Therefore, following the notations of the proof of Corollary  \ref{cor:pre_split_num}, we have $\prod_{i=i}^r\mathrm{Ker}(\tilde{\varphi}^{i}) = \mathrm{Ker}(\varphi)$.
\end{proof}

\begin{example}\label{ex:etale:power}
Let $K = (K_1)^{\oplus r}$ where $K_1/\QQ$ is a Galois CM-field with Galois group $G= G_1$. By Corollary \ref{cor:split:num}  and Theorem \ref{thm:numerator} we know that $\hat{H}^1(\QQ, \mathbf{X}^\star(\T)) = (\ZZ/2\ZZ)^{r-1}$ if the $2$-Sylow subgroups of $G$ are cyclic, otherwise $\hat{H}^1(\QQ, \mathbf{X}^\star(\T)) = (\ZZ/2\ZZ)^r$. Since $K_1$ is Galois, we have $\hat{H}^2(G, \mathbf{X}^\star(\T_1^K)) = \hat{H}^2(G, \mathbf{X}^\star(\T_1))^{\oplus r}  =0$ by Corollary \ref{cor:cohom:ld}. Therefore the cohomology of (\ref{eq:etale:t0}) yields 
\[0\rightarrow \hat{H}^1(G, \mathbf{X}^\star(\T^K))\rightarrow \hat{H}^2(G, \mathbf{X}^\star(\T_1^K))\rightarrow \hat{H}^2(G, \ZZ)\rightarrow \hat{H}^2(G, \mathbf{X}^\star(\T^K))\rightarrow 0.    \]
The first arrow has trivial cokernel if $G_2$ is non-cyclic, else it has cokernel $\ZZ/2\ZZ$. Therefore, we get  $\hat{H}^2(G, \mathbf{X}^\star(\T^K)) = \hat{H}^2(G, \ZZ)$ if $G_2$ is not cyclic, else we get $\hat{H}^2(G, \mathbf{X}^\star(\T^K)) = \hat{H}^2(G, \ZZ)/N_1$. The same logic works for the cohomology of any subgroup of $G$. Therefore, we get $\Sh^2(\T^K) = \Sh^2(\T^{K_1})$. We can conclude
\[\tau(\T^K) = 2^{r-1}\tau(\T^{K_1}).\]
The exact same resonning gives us the following proposition. 
\end{example}

\begin{proposition}\label{prop:split:gal}
Assume $K = \bigoplus_{i=1}^r (K_i)^{\oplus j_i}$ for some $j_1, \cdots, j_r\in \mathbb{N}$ such that each $K_i$ is a Galois CM-field and $\mathrm{Gal}(K/\QQ) = \prod_{i=1}^r\mathrm{Gal}(K_i/\QQ)$. We get 
\[\tau(\T^K) = \prod_{i=1}^r 2^{j_i-1}\tau(\T^{K_i}).\] \end{proposition}